\documentclass[a4paper,11pt]{article}

\usepackage[T1]{fontenc}
\usepackage{lmodern}
\usepackage[utf8]{inputenc}

\usepackage{amssymb}
\usepackage{amsmath}
\usepackage{amsthm}

\usepackage{mathtools,amssymb,amsthm,mathrsfs,calc,graphicx,xcolor,dsfont,tikz,bm}
\usepackage[
  bookmarks=true,
  bookmarksnumbered=true,
  bookmarksopen=true,
  unicode=true,
  pdftitle={Illumination Bodies in Projective Geometries},
  pdfauthor={Rotem Assouline, Florian Besau, Elisabeth M. Werner},
  pdfsubject={MSC 2020: Primary 52A38; Secondary 52A55, 53C20, 53C60},
  pdfcreator={PDFLaTeX},
  pdfkeywords={
    illumination body,
    affine surface area,
    weighted volume derivatives,
    convex geometry in space forms,
    projective Finsler structures,
    variational formulas for volume
  },
  colorlinks=true,colorlinks=true,
  linkcolor=black,
  citecolor=black,
  filecolor=black,
  urlcolor=black
]{hyperref}

\usetikzlibrary{arrows,calc,decorations.pathreplacing,fadings,intersections}
\usetikzlibrary{external}

\numberwithin{equation}{section}
\numberwithin{figure}{section}

\newtheorem{theorem}{Theorem}[section]
\newtheorem{lemma}[theorem]{Lemma}

\newtheorem{proposition}[theorem]{Proposition}
\newtheorem{conjecture}[theorem]{Conjecture}

\theoremstyle{definition}
\newtheorem{example}[theorem]{Example}
\newtheorem{definition}[theorem]{Definition}

\makeatletter
    \def\@cite#1#2{[\textbf{#1}\if@tempswa , #2\fi]}	%text
    \def\@biblabel#1{[#1]}								%bibliography
\makeatother
\def\ba{\mathbf{a}}

\def\bb{\mathbf{b}}

\def\be{\mathbf{e}}

\def\bn{\mathbf{n}}

\def\bp{\mathbf{p}}
\def\bq{\mathbf{q}}

\def\bu{\mathbf{u}}
\def\by{\mathbf{y}}
\def\bz{\mathbf{z}}
\def\bx{\mathbf{x}}

\def\bv{\mathbf{v}}
\def\bp{\mathbf{p}}
\def\bo{\mathbf{o}}

\newcommand{\HH}{\mathbb{H}}

\newcommand{\RR}{\mathbb{R}}
\renewcommand{\SS}{\mathbb{S}}
\newcommand{\Sp}{\mathrm{Sp}}

\newcommand{\cI}{\mathcal{I}}
\newcommand{\cK}{\mathcal{K}}

\newcommand{\dd}{{\mathrm{d}}}
\newcommand{\conv}{\operatorname{conv}}

\newcommand{\vol}{\operatorname{vol}}

\newcommand{\interior}{\operatorname{int}}

\newcommand{\as}{\operatorname{as}}

\begin{document}

\title{Illumination Bodies in Projective Geometries}

\author{Rotem Assouline\footnotemark[1],\; Florian Besau\footnotemark[2],\; Elisabeth M.\ Werner\footnotemark[3]}

\renewcommand{\thefootnote}{\fnsymbol{footnote}}

\footnotetext[1]{%
  Sorbonne Université, Université Paris Cité, CNRS, IMJ-PRG, France.\\ Email: \texttt{assouline@imj-prg.fr}%
}
\footnotetext[2]{%
  Technische Universität Wien, Austria. Email: \texttt{florian.besau@tuwien.ac.at}%
}
\footnotetext[3]{%
  Case Western Reserve University, USA.  Email: \texttt{elisabeth.werner@case.edu}%
}

\date{}
\maketitle

\begin{abstract}
We extend the notion of illumination bodies to Riemannian spaces of constant curvature and to projective Finsler geometries. We prove that the derivative of their volume defines a notion of surface area for convex bodies in these settings, generalizing the affine surface area in Euclidean space.

The proof is based on a general result on the derivative of weighted volumes of weighted illumination bodies in Euclidean space. In the appendix, we give some explicit examples for non-Euclidean illumination bodies.



\end{abstract}

\section{Introduction and main results}

The illumination bodies of a convex body form a one-parameter family of convex bodies defined by an equi-affine construction that relies only on convex hulls and volume. More precisely, a point belongs to the illumination body if and only if the volume of the cap body generated by the convex hull of the point and the convex body is bounded by a prescribed parameter value.

Illumination bodies were introduced in \cite{Werner:1994} as a conceptual dual to the affine construction of floating bodies \cite{BL:1988, SW:1990}. They are among the fundamental affine constructions associated with convex bodies, alongside, for example, John--Loewner ellipsoids, convolution bodies, centroid bodies, and Santaló regions.
See \cite{HL:2022, MW:2019, MW:2020, Zawalski:2025} for some more recent results on classical illumination bodies.

Since illumination bodies associate to each convex body a one-parameter family of convex bodies that converges to the original body in the Hausdorff metric, it is natural to study the corresponding change in volume. It was shown in \cite{Werner:1994} that the (one-sided) derivative of the volume yields an equi-affine invariant notion of surface area, namely the classical affine surface area.

The affine surface area is a fundamental affine invariant associated with convex bodies in Euclidean space. It was first introduced by Blaschke~\cite{Blaschke:1923} for convex bodies in low dimensions under additional regularity assumptions, and later extended to arbitrary dimensions without such assumptions; see, for example, the survey~\cite{SW:2023}. Ludwig and Reitzner~\cite{LR:1999} obtained a characterization of affine surface area as essentially the only upper semicontinuous, equi-affine invariant valuation on convex bodies.

\medskip
In this work, we continue a recent line of research aimed at extending classical affine constructions on convex bodies to more general geometric settings. For example, floating bodies~\cite{BW:2016, BW:2018, BW:2023} have been extended to projective geometries, and centroid bodies~\cite{BHPS:2023} have been extended to spherical spaces.
In \cite{BLW:2018, BRT:2021, BT:2020} random/best approximation of convex bodies in projective geometries were studied. Moreover,  Busemann's intersection inequality has been studied in spherical and hyperbolic spaces in~\cite{DKY:2018}.
The class of $R$-ball convex bodies is an area of intense recent research, e.g., \cite{AAF:2025, AACF:2025, BLN:2026, DT:2023} and analogs of affine surface area and notions of floating bodies \cite{SWY:2025}
have been established in this class.
\medskip

We outline our main results in the following subsections.

\medskip
\noindent\textbf{Acknowledgments.} EMW was supported by NSF grants DMS-2103482 and DMS-2506790. RA is supported by the Rothschild fellowship (Yad Hanadiv) and by the Fondation Sciences Math\'ematiques de Paris (FSMP).
Work for this paper has been started while the authors attended the Dual Trimester Program \emph{Synergies between modern probability, geometric analysis
and stochastic geometry} at the Hausdorff Research Institute for Mathematics in Bonn and was also continued while the last author was Research Fellow at  the University of M\"unster. All support is gratefully acknowledged.

\medskip

\subsection{Illumination bodies in Riemannian space forms}

A \emph{real space form} is a simply connected, complete Riemannian manifold of constant sectional curvature $\lambda \in \mathbb{R}$. For $\lambda \in \mathbb{R}$ and $n \geq 2$, we denote by $\Sp^n(\lambda)$ the $n$-dimensional real space form of curvature $\lambda$. Important examples include the sphere $\mathbb{S}^n = \Sp^n(1)$, hyperbolic space $\mathbb{H}^n = \Sp^n(-1)$, and Euclidean space $\mathbb{R}^n = \Sp^n(0)$.

The space form $\Sp^n(\lambda)$ carries a natural Riemannian volume form, which induces the \emph{Riemannian volume measure} $\vol_n^\lambda$.

A compact, geodesically convex set $K \subset \Sp^n(\lambda)$ with nonempty interior is called a \emph{convex body}. For $\lambda > 0$, every convex body is contained in an open hemisphere of $\Sp^n(\lambda)$. We denote the family of convex bodies in $\Sp^n(\lambda)$ by $\mathcal{K}_0(\Sp^n(\lambda))$. For $\bx \in \Sp^n(\lambda)$ and $K \in \mathcal{K}_0(\Sp^n(\lambda))$, we set
\[
    [\bx, K] := \conv(\{\bx\} \cup K).
\]

\begin{definition}[Illumination body in $\Sp^n(\lambda)$]
Let $\lambda \in \mathbb{R}$ and $n\geq 2$. For $K \in \mathcal{K}_0(\Sp^n(\lambda))$ and $\delta > 0$, the \emph{illumination body} $\mathcal{I}_\delta^\lambda(K)$ is defined by
\[
    \mathcal{I}_\delta^\lambda(K)
    = \left\{ \bx \in \Sp^n(\lambda) : \vol_n^\lambda\bigl([\bx, K] \setminus K\bigr) \leq \delta \right\}.
\]
\end{definition}

The \emph{floating area}
\begin{equation}\label{eqn:floating_area}
    \Omega^\lambda(K)
    = \int_{\partial K} H_{n-1}^{\lambda}(K,\bx)^{\frac{1}{n+1}} \,\vol_{\partial K}^\lambda(\dd\bx)
\end{equation}
of a convex body $K \in \mathcal{K}_0(\Sp^n(\lambda))$ was introduced in \cite{BW:2016,BW:2018} as a natural surface area arising from the derivative of the volume of the non-Euclidean floating body of $K$. Here, $\vol_{\partial K}^\lambda$ denotes the boundary measure induced by the Riemannian volume form on $\Sp^n(\lambda)$, restricted to the (Lipschitz) submanifold $\partial K$, and $H_{n-1}^{\lambda}(K,\bx)$ denotes the (generalized) Gauss--Kronecker curvature of $\partial K$ at $\bx$. For $\lambda = 0$, the quantity \eqref{eqn:floating_area} coincides with the classical affine surface area.

The derivative of the volume of the illumination body also gives rise to the floating area.

\begin{theorem}\label{thm:main_riemannian}
Let $n \geq 2$ and $\lambda \in \mathbb{R}$. If $K \in \mathcal{K}_0(\Sp^n(\lambda))$, then the right derivative of $\vol_n^\lambda(\mathcal{I}_\delta^\lambda(K))$ at $\delta = 0$ exists. More precisely,
\[
    \lim_{\delta \to 0^+}
    \frac{\vol_n^\lambda(\mathcal{I}_\delta^\lambda(K)) - \vol_n^\lambda(K)}
         {\delta^{\frac{2}{n+1}}}
    = c_n \, \Omega^\lambda(K),
\]
where
\begin{equation}\label{def:cn}
    c_n
    = \frac{1}{2}
    \left( \frac{n(n+1)}{\vol_{n-1}(B^{n-1}_{2})} \right)^{\frac{2}{n+1}}.
\end{equation}
\end{theorem}

This result extends the Euclidean case $\lambda = 0$, first established in \cite{Werner:1994}.

\medskip
In Euclidean space, illumination bodies are convex.
In general, this fails for $\lambda > 0$. We refer to Theorem~\ref{thm:spherical_not_convex} where we show that the illumination body of a spherical convex body $K\subset \SS^2$ is not convex for any $\delta>0$ if the boundary of $K$ contains a geodesic segment. By contrast, in the hyperbolic plane we can show the following.

\begin{theorem}\label{thm:convex_hyperbolic_plane}
    The illumination body $\cI_\delta^h(K)$ of a convex body $K\subset \HH^2$ is convex for all $\delta>0$.
\end{theorem}

We conjecture that this property persists in space forms of nonpositive curvature also in higher dimensions.

\begin{conjecture}\label{conj:hyperbolic_convex}
Let $n \geq 3$ and $\lambda \leq 0$. If $K \in \mathcal{K}_0(\Sp^n(\lambda))$, then $\mathcal{I}_\delta^\lambda(K) \in \mathcal{K}_0(\Sp^n(\lambda))$ for all $\delta > 0$.
\end{conjecture}

\subsection{Illumination bodies in projective Finsler geometries}

Projective Finsler geometries can be viewed as generalizations of hyperbolic space forms $\Sp^n(\lambda)$ for $\lambda < 0$. A projective Finsler geometry $(X,F)$ is defined on an open convex set $X \subset \RR^n$, where the Finsler metric $F: TX \to \RR$ is a continuous function on the tangent bundle such that, for every $\bx \in X$, the map $F(\bx,\cdot): T_{\bx}X \to \RR$ defines a norm on the tangent space $T_{\bx}X \cong \RR^n$.

We assume that $F$ is $C^3$-smooth away from the zero section of $TX$ and strongly convex, that is, the vertical Hessian
\[
    \frac{\partial^2 F}{\partial v_i \partial v_j}(\bx,\bv)
\]
is nondegenerate for every $\bv \neq 0$.

Given a Finsler geometry $(X,F)$ and a smooth curve $\gamma:[a,b]\to X$, the $F$-length of $\gamma$ is defined by
\[
    \mathrm{Len}^F(\gamma) = \int_a^b F(\gamma(t),\gamma'(t))\, \dd t.
\]
The associated (generally non-symmetric) distance function $d_F$ is given by
\[
    d_F(\bx,\by) = \inf\{ \mathrm{Len}^F(\gamma) : \text{$\gamma$ is a smooth curve from $\bx$ to $\by$} \}.
\]
In general, $d_F(\bx,\by) \neq d_F(\by,\bx)$.

We assume that $F$ is \emph{projective}, meaning that straight line segments are (unparametrized) geodesics. Under the above regularity assumptions, these are the only geodesics. We denote by $\cK_0(X)$ the family of convex bodies in $(X,F)$; note that $K \in \cK_0(X)$ if and only if $K \in \cK_0(\RR^n)$ and $K \subset X$. An important example is given by Hilbert geometries, in which the induced distance $d_F$ is a metric; see Section~\ref{sec:Hilbert_geometries}.

Unlike the Riemannian setting, a Finsler geometry does not admit a canonical volume form. Instead, a choice of volume is determined by prescribing densities on the tangent spaces $T_{\bx}X$; see \cite[Sec.\ 5.5.3]{BBI:2001} and also \cite{APT:2004} for an axiomatic treatment. Classical choices include the \emph{Busemann}, \emph{Holmes--Thompson}, and \emph{Gromov mass and comass} definitions; for more details see Section~\ref{sec:Background_Finsler}. Any such choice induces a volume measure $\vol_n^F$ on $X$ with continuous density $\varphi_F : X \to (0,\infty)$ with respect to Lebesgue measure, see \cite[Prop.\ 5.5.11]{BBI:2001}.

\begin{definition}[Finsler illumination body]
Let $(X,F)$ be a projective Finsler geometry equipped with a volume measure $\vol_n^F$, and let $K \in \cK_0(X)$. For $\delta > 0$, the \emph{illumination body} $\cI_\delta^F(K)$ is defined by
\[
    \cI_\delta^F(K)
    = \left\{ \bx \in X : \vol_n^F([\bx,K] \setminus K) \leq \delta \right\}.
\]
\end{definition}

In \cite{BLW:2018}, the volume derivative of floating bodies in Hilbert geometries was studied, leading to a notion of surface area extending the hyperbolic floating area. In a similar spirit, the derivative of the volume of illumination bodies in projective Finsler geometries gives rise to a corresponding surface area.

\begin{theorem}\label{thm:main_finsler}
Let $(X,F)$ be a projective Finsler geometry with volume measure $\vol_n^F$, and let $K \in \cK_0(X)$. Then
\[
    \lim_{\delta \to 0^+}
    \frac{\vol_n^F(\cI_\delta^F(K)) - \vol_n^F(K)}
         {\delta^{\frac{2}{n+1}}}
    = c_n \, \Omega^F(K),
\]
where $c_n$ is defined in \eqref{def:cn} and
\begin{equation}\label{def:finsler-affine-surface-area}
    \Omega^F(K)
    = \int_{\partial K}
    H_{n-1}(K,\bx)^{\frac{1}{n+1}}
    \varphi_F(\bx)^{\frac{n-1}{n+1}}
    \,\mathcal{H}^{n-1}(\dd \bx),
\end{equation}
with $\varphi_F$ denoting the density of $\vol_n^F$ and $\mathcal{H}^{n-1}$ is $(n-1)$-dimensional Hausdorff measure in $\RR^n$.
\end{theorem}

If $(X,F)$ is a Hilbert geometry, then $\Omega^F$ coincides with the surface area previously obtained in \cite{BLW:2018} via floating bodies and random or best approximating polytopes.

\subsection{Dual volumes and illumination bodies}

The classical Brunn--Minkowski theory (BMt) of convex bodies is fundamentally based on volume and Minkowski addition. Lutwak~\cite{Lutwak:1975_1} introduced the dual Brunn--Minkowski theory (dBMt), obtained by replacing Minkowski addition with radial addition. This theory exhibits many striking parallels with the classical Brunn--Minkowski theory. We refer to the monograph by Schneider \cite{Schneider:2014} and the recent survey \cite{HYZ:2025} for the history of the dBMt and for more recent results on inequalities for dual volumes see \cite{SZ:2025,XZ:2022}.

\medskip
In BMt the quermassintegrals (intrinsic volumes) are a fundamental notion that arise, for example, by taking the average volumes of projections into $k$-dimensional linear subspaces for $k=0,\dotsc,n$.
The quantities dual to the quermassintegrals in the dBMt are the \emph{dual volumes} $\widetilde{V}_k$ for $k=0,\dotsc,n$, that arise from taking the average volumes over the intersection with all $k$-dimensional linear subspaces.
As an extension, Lutwak \cite{Lutwak:1975_2} derived that for $q \in \RR$ and a star body $L \subset \RR^n$ the dual volume of $L$ may be expressed by
\[
    \widetilde{V}_q(L)
    = \frac{1}{n} \int_{\SS^{n-1}} \rho(L,\bu)^q \, \mathcal{H}^{n-1}(\dd \bu),
\]
where $\rho(L,\cdot)$ denotes the (continuous) radial function of $L$. Note that $\widetilde{V}_n = \vol_n$.

We show that the derivative of the dual volume along illumination bodies exists and admits an explicit representation.

\begin{theorem}\label{thm:main_dual}
Let $n \geq 2$, $q \in \RR \setminus \{0\}$, and let $K \in \cK_0(\RR^n)$ with $\bo \in \interior K$. Then
\[
    \lim_{\delta \to 0^+}
    \frac{\widetilde{V}_q(\cI_\delta(K)) - \widetilde{V}_q(K)}
         {\delta^{\frac{2}{n+1}}}
    = c_n \frac{q}{n}
    \int_{\partial K}
    H_{n-1}(K,\bx)^{\frac{1}{n+1}}
    \|\bx\|^{q-n}
    \, \mathcal{H}^{n-1}(\dd \bx),
\]
where $\|\cdot\|$ denotes the standard Euclidean norm in $\RR^n$.
\end{theorem}

For $q = n$, this reduces to the classical result yielding affine surface area.

\subsection{Weighted illumination bodies in Euclidean space}

We now turn to the main result of this paper, which will be used in Section~\ref{sec:applications} to establish Theorems~\ref{thm:main_riemannian},~\ref{thm:main_finsler}, and~\ref{thm:main_dual} as special cases. It generalizes the result of \cite{Werner:1994} on the volume derivative of illumination bodies. Also note that (weighted) illumination surface bodies were previously studied in \cite[Sec.\ 3]{WY:2010}.

\medskip
To describe our setting, we define weighted illumination bodies as follows. They are the analog of weighted floating bodies which were  introduced in \cite{Werner:2002}.

\begin{definition}[weighted illumination body]\label{def:WIB}
Let $K \in \cK_0(\RR^n)$ and let $U \supset K$ be an open set. Let $\varphi : U \to (0,\infty)$ be a continuous and integrable function. Let $A$ be a measurable set in $\mathbb{R}^n$. We define a measure $\vol_n^\varphi$ on $\RR^n$ by
\[
    \vol_n^\varphi(A)
    = \int_{A \cap U} \varphi \, \dd \lambda_n,
\]
where $\lambda_n$ denotes Lebesgue measure.

For $\delta > 0$, the \emph{$\varphi$-weighted illumination body} $\cI_\delta^\varphi(K)$ is defined by
\[
    \cI_\delta^\varphi(K)
    = \left\{ \bx \in \RR^n : \vol_n^\varphi([\bx,K] \setminus K) \leq \delta \right\}.
\]
\end{definition}

Equivalently, we may express the $\varphi$-weighted illumination body as sublevel sets of a functional $V_K^\varphi : \RR^n \to [0,+\infty)$ defined by
\begin{equation*}
    V_K^\varphi(\bz) := \vol_n^\varphi([\bz,K]\setminus K).
\end{equation*}
Then by Definition~\ref{def:WIB}, the $\varphi$-weighted illumination bodies $\cI_\delta^\varphi(K)$ is
\begin{equation*}
    \cI_\delta^\varphi(K) = \{\bz \in \mathbb{R}^n : V_K^\varphi(\bz) \leq \delta\}.
\end{equation*}

In the special case $U=\RR^n$ and $\varphi\equiv 1$ we have, see \cite{MW:2020},
\begin{equation}\label{eqn:VK_uniform}
    V_K(\bz) = -\frac{1}{2} \vol_n(K) + \frac{1}{2n} \int_{\partial K} \left|(\bz-\bx)\cdot \bn_K(\bx)\right|\, \mathcal{H}^{n-1}(\dd \bx),
\end{equation}
where $\bn_K(\bx)$ denotes the outward unit normal of $\partial K$ at $\bx$, which exists for almost all boundary points $\bx$.
Convexity of $V_K$ follows by the convexity of the map $\bz\mapsto |(\bz-\bx)\cdot \bn_K(\bx)|$ which is defined for almost all $\bx\in\partial K$.
In particular, this shows that the classical illumination body $\cI_\delta(K) = \{\bz: V_K(\bz)\leq \delta\}$ is convex for all $\delta>0$ as sublevel sets of a convex function.

\medskip
For non-uniform weight functions $V_K^\varphi$ is in general not convex, but can be expressed similar to \eqref{eqn:VK_uniform} as follows.

\begin{theorem}\label{thm:VK_weighted}
    Let $K\in\cK_0(\RR^n)$, $U\supset K$ be an open set, and let $\varphi:U\to (0,+\infty)$ be a continuous and integrable function. Then
    \begin{equation*}
        V_K^\varphi(\bz) = -\frac{1}{2}\vol_n^\varphi(K)+ \frac{1}{2} \int_{\partial K} \!\!\!\! |(\bz-\bx)\cdot \bn_K(\bx)| \left(\!\int_{0}^1 \varphi([\bz,\bx]_s)\, s^{n-1}  \dd s\! \right)\!\mathcal{H}^{n-1}(\dd \bx),
    \end{equation*}
    were $[\bz,\bx]_s = (1-s)\bz+s\bx$ and we extend $\varphi$ to $\RR^n$ by setting $\varphi(\by)=0$ if $\by\not\in U$.
    In particular, this shows that $V_K^\varphi$ is continuous.
\end{theorem}

\medskip
In general, the $\varphi$-illumination body need not be convex or bounded. However, if $K$ contains the origin in its interior, then $\cI_\delta^\varphi(K)$ is star-shaped with respect to the origin and converges to $K$ in the Hausdorff metric as $\delta \to 0^+$. In particular, for sufficiently small $\delta$ we have
\[
    K \subset \cI_\delta^\varphi(K) \subset U,
\]
so that $\cI_\delta^\varphi(K)$ is a compact star-shaped set, that is, a \emph{star body}.
Let $\psi : U \to (0,\infty)$ be another continuous and integrable function. Then $\vol_n^\psi(\cI_\delta^\varphi(K) \setminus K)$ is finite for all sufficiently small $\delta$. Our main result shows that this quantity admits an explicit right derivative at $\delta = 0$.

\begin{theorem}\label{thm:main_weighted}
Under the above assumptions,
\begin{align*}
    &\lim_{\delta \to 0^+}
    \frac{\vol_n^{\psi}(\cI_\delta^\varphi(K)) - \vol_n^{\psi}(K)}
         {\delta^{\frac{2}{n+1}}}\\
    &\qquad \qquad \qquad \qquad \qquad = c_n
    \int_{\partial K}
    H_{n-1}(K,\bx)^{\frac{1}{n+1}}
    \varphi(\bx)^{-\frac{2}{n+1}}
    \psi(\bx)
    \, \mathcal{H}^{n-1}(\dd \bx),
\end{align*}
where $c_n$ is defined in \eqref{def:cn}.
\end{theorem}

\subsection{Outline of the paper}

In the next section we recall some basic notation, before proving our main results, Theorems~\ref{thm:VK_weighted} and~\ref{thm:main_weighted}, in Section~\ref{sec:weighted}. We explain how this can be used to derive Theorems~\ref{thm:main_riemannian},~\ref{thm:main_finsler}, and~\ref{thm:main_dual} in Section~\ref{sec:applications}. In Subsection~\ref{sec:Riemannian} we also prove Theorem~\ref{thm:convex_hyperbolic_plane}. In the appendix we give a few explicit examples for the illumination bodies in non-Euclidean spaces.

\section{Preliminaries}

We collect notation and background material used throughout the paper.
We refer to \cite{Gardner:2006,Schneider:2014} for general background on the Brunn--Minkowski theory of convex bodies.

The Euclidean norm of $\bx \in \RR^n$ is denoted by $\|\bx\|$.
The Euclidean unit ball is
\[
    B_2^n = \{\bx \in \RR^n : \|\bx\| \leq 1\},
\]
and the Euclidean unit sphere is $\SS^{n-1} = \partial B_2^n$.
For $r > 0$ and $\bx \in \RR^n$, the closed ball of radius $r$ centered at $\bx$ is
\[
    B_r(\bx) = \{\by \in \RR^n : \|\by - \bx\| \leq r\}.
\]

The Lebesgue measure in $\RR^n$ is denoted by $\lambda_n$, and the Hausdorff measure by $\mathcal{H}^{n-1}$. Integration with respect to these measures will usually be written simply as $\dd\bx$.

\medskip

A convex body $K \subset \RR^n$ containing the origin in its interior is uniquely determined by its \emph{radial function} $\rho(K,\cdot) : \SS^{n-1} \to (0,\infty)$ defined by
\[
    \rho(K,\bu) = \max \{ R \geq 0 : R\bu \in K \}, \qquad \bu \in \SS^{n-1}.
\]
The function $\rho(K,\cdot)$ is continuous. More generally, every continuous function $\rho : \SS^{n-1} \to (0,\infty)$ determines a \emph{star body} $L \subset \RR^n$ such that $\rho = \rho(L,\cdot)$, that is, $L$ is compact, contains the origin $\bo$ in its interior, and satisfies $[\bo,\bx] \subset L$ for all $\bx \in L$.

\medskip

For $\bx \in \partial K$, we denote by $H_{n-1}(K,\bx)$ the Gauss--Kronecker curvature of $\partial K$ at $\bx$, and by $\bn_K(\bx)$ the Euclidean outer unit normal.

\subsection{Riemannian space forms and projective models}

We use a projective model for the space forms $\Sp^n(\lambda)$ of constant curvature $\lambda \in \RR$. Set
\[
    \mathbb{B}_\lambda^n :=
    \begin{cases}
        \interior\, B_{\bo}\!\left(1/\sqrt{-\lambda}\right) & \text{if } \lambda < 0,\\
        \RR^n & \text{if } \lambda \geq 0,
    \end{cases}
\]
and define a Riemannian metric $g^\lambda$ on $\mathbb{B}_\lambda^n$ by
\begin{equation}
    g^\lambda(X_{\bp},Y_{\bp})
    = \frac{X_{\bp}\cdot Y_{\bp}}{1+\lambda\|\bp\|^2}
    - \lambda \frac{(X_{\bp}\cdot \bp)(Y_{\bp}\cdot \bp)}{(1+\lambda\|\bp\|^2)^2},
\end{equation}
for $\bp \in \mathbb{B}_\lambda^n$.

Then $(\mathbb{B}_\lambda^n ,g^\lambda)$ is a Riemannian manifold of constant sectional curvature $\lambda$. By the Killing--Hopf theorem, there exists, up to isometry, a unique simply connected complete Riemannian manifold $\Sp^n(\lambda)$ of constant sectional curvature $\lambda$; see, e.g., \cite[Ch.~6]{KN:1963}, \cite[Thm.~1.9]{Lee:1997}, or \cite[Ch.~8, Cor.~25]{ON:1983}.

Thus, for $\lambda \leq 0$, $(\mathbb{B}_\lambda^n, g^\lambda)$ is isometric to $\Sp^n(\lambda)$, while for $\lambda > 0$ it is isometric to an open hemisphere.

Moreover, $(\mathbb{B}_\lambda^n, g^\lambda)$ is \emph{projective}, i.e., its (unparametrized) geodesics are straight line segments contained in $\mathbb{B}_\lambda^n$. Hence, geodesically convex bodies in $(\mathbb{B}_\lambda^n, g^\lambda)$ coincide with Euclidean convex bodies contained in $\mathbb{B}_\lambda^n$.

The Riemannian volume measure $\vol_n^\lambda$ has density
\begin{equation}
    \varphi_\lambda(\bp)
    = (1+\lambda\|\bp\|^2)^{-\frac{n+1}{2}},
\end{equation}
see \cite{BW:2018}.

Thus, for $\lambda \leq 0$, the $\lambda$-illumination body satisfies
\begin{equation}\label{eqn:space_form_illumination_is_weighted}
    \cI^\lambda_{\delta}(K) = \cI^{\varphi_\lambda}_\delta(K).
\end{equation}
For $\lambda > 0$, this identity holds as long as $\cI^{\varphi_\lambda}_\delta(K)$ is bounded, which is the case for all sufficiently small $\delta$.

We also recall from \cite{BW:2018} that
\begin{equation}\label{eqn:lboundary}
	\vol_{\partial K}^\lambda(\dd\bx)
    = \sqrt{\frac{1+\lambda\, (\bn_K^e(\bx)\cdot \bx)^2}{(1+\lambda\, \|\bx\|^2)^{n}}} \, \dd\bx,
\end{equation}
and for normal boundary points $\bx \in \partial K$,
\begin{equation}\label{eqn:lgauss}
	H_{n-1}^\lambda(K,\bx)
    = H_{n-1}^e(K,\bx)
    \left(\frac{1+\lambda\|\bx\|^2}{1+\lambda\, (\bn_K^e(\bx)\cdot \bx)^2}\right)^{\frac{n+1}{2}}.
\end{equation}
Here $\bn_K^e(\bx)\in\SS^{n-1}$ denotes the outward unit normal to $\partial K$ at $\bx$ in the Euclidean space $\RR^n$, and $H_{n-1}^\lambda(K,\cdot)$, respectively $H_{n-1}^e(K,\cdot)$, denotes the (generalized) Gauss--Kronecker curvature of $\partial K$ as hypersurface of $(\mathbb{B}^n_\lambda,g^\lambda)$, respectively the Euclidean space  $\RR^n$.

\subsection{Notions of volume in Finsler geometries}\label{sec:Background_Finsler}

We recall some standard definitions of volume on a Finsler geometry $(X,F)$.

\subsubsection{Busemann definition of volume}
In the Busemann definition \cite{Busemann:1946} one chooses a density on the normed vector space $(T_{\bp}X, \|\cdot\|_{\bp})$ such that the unit ball $B(\bp)=\{X_{\bp}\in T_{\bp}X: \|X_{\bp}\|_{\bp}\leq 1\}$ with respect to $\|\cdot\|_{\bp}=F(\bp,\cdot)$ has volume $\vol_n(B_2^n)$. Thus
\begin{equation*}
    \varphi_{\mathrm{B}}(\bp) = \frac{\vol_n(B_2^n)}{\vol_n(B_{\bp})}.
\end{equation*}
This definition of volume coincides with the $n$-dimensional Hausdorff measure, see \cite[Ex.\ 5.5.12]{BBI:2001}.

\subsubsection{Holmes--Thompson definition of volume}
In the Holmes--Thompson definition \cite{HT:1979} we normalize the dual unit ball $B_{\bp}^* = \{X_{\bp}^*\in (T_{\bp}X)^* : X_{\bp}^*(X_{\bp})\leq 1 \text{ for all $X_{\bp}\in B_{\bp}$} \}$ and therefore
\begin{equation*}
    \varphi_{\mathrm{HT}}(\bp) = \frac{\vol_n(B_{\bp}^*)}{\vol_n(B_2^n)}.
\end{equation*}
This is related to the natural symplectic structure on $X\times X^*$ and the symplectic volume of $B_{\bp}\times B_{\bp}^*$; see \cite{AP:2005}.

\subsubsection{Gromov mass and comass definition of volume}
In the Gromov mass definition \cite{Gromov:1983} one normalizes by the maximum of the volume of the affine images of cross-polytopes that are contained in $B_{\bp}$, and in the Gromov comass definition one normalizes by the minimum of the volume of the affine images of cubes that contain $B_{\bp}$. Thus, if $P=\mathrm{conv}\{\pm \be_1,\dotsc,\pm\be_n\}$ denotes the cross-polytope in $\RR^n$, then
\begin{equation*}
    \varphi_{\mathrm{mass}}(\bp) = \frac{\vol_n(P)}{\max\{ \vol_n( A P) : A\in\mathrm{GL}_n \text{ and $AP\subset B_{\bp}$}\}}
\end{equation*}
and, if $C=[-1,1]^n$ denotes the cube, then
\begin{equation*}
    \varphi_{\mathrm{comass}}(\bp) = \frac{\vol_n(C)}{\min\{ \vol_n( A C) : A\in\mathrm{GL}_n \text{ and $AC\supset B_{\bp}$}\}}
\end{equation*}
Similarly, one could choose ellipsoids instead of cross-polytopes, or cubes, and obtain minimal and maximal Riemannian volumes, see \cite[Ex.\ 5.5.15]{BBI:2001}.

\subsection{Hilbert geometries}\label{sec:Hilbert_geometries}

Hilbert geometries are classical examples of projective Finsler geometries defined on an open and convex domain $X\subset \RR^n$; we refer to \cite{Papadopoulos:2014, PT:2014} for more details and background on Hilbert and Finsler geometries.

The Finsler norm $H$ on $TX$ is defined by the harmonic mean
\begin{equation*}
    H(\bp, X_{\bp}) = \frac{1}{2} \left(t_+(\bp,X_{\bp})^{-1} + t_-(\bp,X_{\bp})^{-1}\right)
\end{equation*}
for all $\bp\in X$ and  $X_{\bp}\in T_{\bp}X$, where
\begin{equation*}
    t_{\pm}(\bp,X_{\bp}) = \sup\{r\geq 0: \bp\pm rX_{\bp} \in X\},
\end{equation*}
i.e., the line through $\bp$ in direction $X_{\bp}$ intersects $\partial X$ in the points $\ba,\bb$
such that $(\ba,\bp,\bb)$ are exactly in this order on the line segment, then $t_+(\bp,X_{\bp}) = \|\bb-\bp\|$ and $t_-(\bp,X_{\bp})=\|\ba-\bp\|$; see also Figure~\ref{fig:hilbert_norm}.
The norm $\|\cdot\|_{\bp}=H(\bp,\cdot)$ defined on $T_{\bp} X$ can also be expressed by
\begin{equation*}
    \|X_{\bp}\|_{\bp} = \max \{ Y_{\bp}^*(X_{\bp}) : Y_{\bp}^* \in D\left[(X-\bp)^*\right]\},
\end{equation*}
where $D[K] = \frac{1}{2} \{\bx-\by: \bx,\by\in K\}$ denotes the difference body of a convex body, $K^*$ is the dual body that is
$K^*= \{\by^*\in(\RR^n)^* : \by^*(\bx)\leq 1 \text{ for all $\bx\in K$}\}$, and we identify $(\RR^n)^*$ with $(T_{\bp} X)^*$; see \cite{PT:2009}.

For $\bp,\bq\in X$ the induced distance function is
\begin{equation*}
    d(\bp,\bq) = \frac{1}{2} \log\left( \frac{\|\ba-\bp\|}{\|\ba-\bq\|} \frac{\|\bb-\bp\|}{\|\bb-\bq\|}\right),
\end{equation*}
where $\ba,\bb$ are the intersection points of $\partial X$ with the line spanned by $\bp,\bq$ and the points $(\ba,\bp,\bq,\bb)$ are exactly in this order on the line segment.

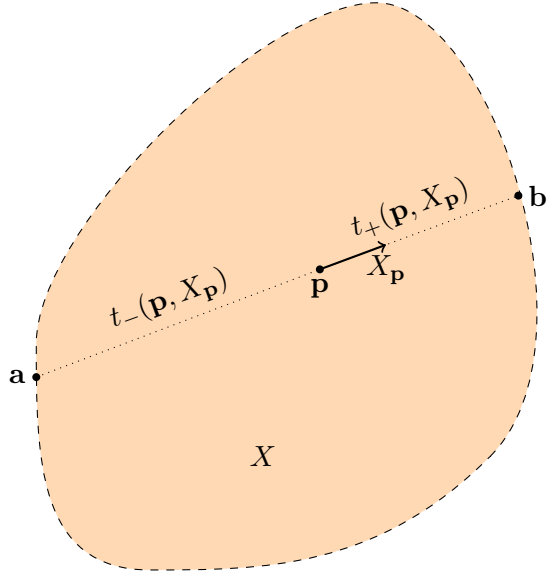
\begin{figure}[t]
\centering
\begin{tikzpicture}[scale=1.5]

% convex domain
\fill[orange,opacity=.3]
  (2,0)
  .. controls (3,1) and (2,4) .. (1,4)
  .. controls (0,4) and (-2,2) .. (-2,1)
  .. controls (-2,0) and (-2,-1) .. (-1,-1)
  .. controls (0,-1) and (1,-1) .. cycle;

% Boundary of the convex domain
\draw[dashed]
  (2,0)
  .. controls (3,1) and (2,4) .. (1,4)
  .. controls (0,4) and (-2,2) .. (-2,1)
  .. controls (-2,0) and (-2,-1) .. (-1,-1)
  .. controls (0,-1) and (1,-1) .. cycle;

% Line through p
\draw[dotted] (-2.0,0.7)
    --(0.5,1.65) node[midway,above, rotate=20] {$t_-(\bp,X_{\bp})$}
    --(2.25,2.3) node[midway,above, rotate=20] {$t_+(\bp,X_{\bp})$};

\draw[thick, ->] (0.5,1.65) --++ ({1.75/3},{0.65/3})
    node[below] {$X_{\bp}$};

% Interior point p
\fill (0.5,1.65) circle (1pt) node[below] {$\bp$};

% Intersection points (chosen consistently with the line)
\fill (-2.0,0.7) circle (1pt) node[left] {$\ba$};

\fill (2.25,2.3) circle (1pt) node[right] {$\bb$};

\node at (0,0) {$X$};
\end{tikzpicture}
\caption{In a Hilbert geometry $(X,H)$ the Finsler norm $H(\bp,X_{\bp})=\|X_{\bp}\|_{\bp}$ for $X_{\bp}\in T_{\bp}(X)$ is defined as the harmonic mean of the distances $t_+(\bp,X_{\bp})$ and $t_-(\bp,X_{\bp})$ to the boundary $\partial X$ on the line through $\bp$ in direction $X_{\bp}$.}
\label{fig:hilbert_norm}
\end{figure}

\section{Proof of the main results} \label{sec:weighted}

\subsection{Proof of Theorem~\ref{thm:VK_weighted}}
    \begin{proof}%[Proof of Theorem~\ref{thm:VK_weighted}]

For $K\in\cK_0(\RR^n)$ and $\bz\not\in K$ we partition $\partial K$ into
the \emph{frontside boundary}
\begin{equation*}
    \partial K^+(\bz) := (\interior [\bz,K]) \cap \partial K,
\end{equation*}
and the \emph{backside boundary}
\begin{equation*}
    \partial K^-(\bz) := (\partial [\bz,K])\cap \partial K,
\end{equation*}
see Figure~\ref{fig:partial_K}.

We may parametrize $\by\in (\interior [\bz,K])\setminus K$ by $\by(\bx,s)=[\bz,\bx]_s=(1-s)\bz+s\bx$ for $s\in(0,1)$ and $\bx\in K^+(\bz)$.
The Jacobian is $J\by(\bx,s)=s^{n-1} [(\bz-\bx)\cdot\bn_K(\bx)]$ for almost all $\bx\in K^+(\bz)$.
Using this we find
\begin{align*}
    V_K^\varphi(\bz)
    &= \vol_n^\varphi([\bz,K]\setminus K)
    = \int_{([\bz,K]\setminus K)\cap U} \varphi\, \dd\lambda_n\\
    &= \int_{\partial K^+(\bz)} |(\bz-\bx)\cdot \bn_K(\bx)| \left(\int_{0}^{1} \mathbf{1}_{U}([\bz,\bx]_s) \varphi([\bz,\bx]_s) s^{n-1}\, \dd s\right) \,\dd \bx,
\end{align*}
where we note that $[\bz,\bx]_s \not \in K$ for $s\in(0,1)$ since $\bx\in\partial K^+(\bz)$.

Similarly, by extending the rays from $\bz$ to meet the backside, we also find
\begin{align*}
    &V_K^{\varphi}(\bz)+\vol_n^\varphi(K) = \vol_n^\varphi([\bz,K])\\
    &\quad = \int_{\partial K^-(\bz)} |(\bx-\bz)\cdot\bn_K(\bx)| \left(\int_{0}^1 \mathbf{1}_{U}([\bz,\bx]_s) \varphi([\bz,\bx]_s) s^{n-1}\, \dd s\right)\dd\bx.
\end{align*}
Thus
\begin{align*}
    V_K^\varphi(\bz) &= -\frac{1}{2} \vol_n^\varphi(K) + \frac{1}{2} \int_{\partial K^+(\bz)\cup \partial K^-(\bz)}  \left|(\bx-\bz)\cdot\bn_K(\bx)\right|\\
    &\qquad\qquad\qquad\qquad\qquad \times \left(\int_{0}^1 \mathbf{1}_{U}([\bz,\bx]_s)  \varphi([\bz,\bx]_s) s^{n-1}\, \dd s\right)\dd\bx.
\end{align*}
This completes the proof, since $\partial K = \partial K^+(\bz)\cup \partial K^-(\bz)$ and $\bz\mapsto |(\bz-\bx)\cdot \bn_K(\bx)|$, as well as $\bz \mapsto \int_{0}^1 \mathbf{1}_{U}([\bz,\bx]_s)  \varphi([\bz,\bx]_s) s^{n-1}\, \dd s$, are continuous functions.
\end{proof}

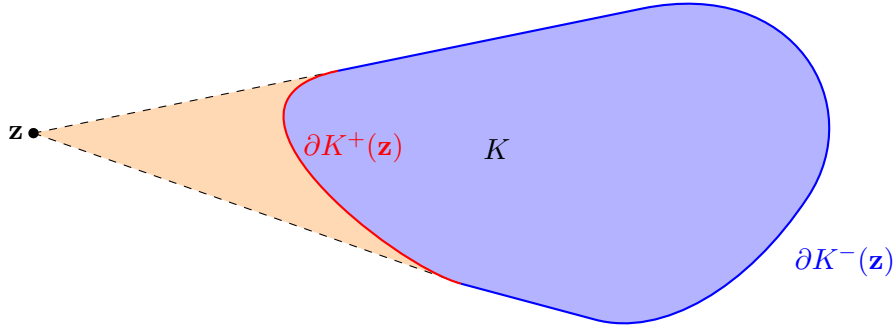
\begin{figure}[t]
\centering
\begin{tikzpicture}[scale=0.65, rotate=-60]

\fill[blue,opacity=.3]
(7,3)
  .. controls (8,4)  and (8,4)  .. (9,5)
  .. controls (10,6) and (10,8) .. (9,10)
  .. controls (8,12) and (5,12) .. (4,9)
  .. controls (3,6) and (3,6) .. (2,3)
  .. controls (1,0)  and (6,2)  .. cycle;

\fill[orange,opacity=0.3]
    (0,-3) --
        (2,3) .. controls (1,0)  and (6,2)  .. (7,3)
        --cycle;

\draw[dashed] (0,-3) -- (7,3);
\draw[dashed] (0,-3) -- (2,3);

% Boundary of the convex domain
\draw[thick, blue]
  (7,3)
  .. controls (8,4)  and (8,4)  .. (9,5)
  .. controls (10,6) and (10,8) .. (9,10)
  .. controls (8,12) and (5,12) .. (4,9)
  .. controls (3,6) and (3,6) .. (2,3);

\draw[thick,red]
(2,3) .. controls (1,0)  and (6,2)  .. (7,3);

\fill (0,-3) circle (3pt) node[left] {$\bz$};

\node at (5,5) {$K$};
\node[red] at (3.5,2.5) {$\partial K^+(\bz)$};
\node[blue] at (10.5,10) {$\partial K^-(\bz)$};

\end{tikzpicture}
\caption{The frontside boundary $\partial K^+(\bz)$ (in red) and backside boundary $\partial K^-(\bz)$ (in blue) as seen from $\bz\not\in K$ partition the boundary of $K$.}
\label{fig:partial_K}
\end{figure}

\subsection{Proof of Theorem ~\ref{thm:main_weighted}}
The proof of Theorem~\ref{thm:main_weighted} follows by establishing the following three lemmas.

First, we have the following classical result, see, for example, \cite[Lem.\ 4.2]{BFH:2010} or \cite[Lem.\ 4.4]{WY:2010} for similar statements.
\begin{lemma}\label{lem:voldiff}
    Let $K\in\cK_0(\RR^n)$ such that $\bo\in\interior K$ and let $L\subset \RR^n$ be a star body such that $K\subset L$.
    Furthermore, let $U\subset L$ be open and let $\psi:U\to [0,+\infty)$ be continuous and integrable.
    Then
    \[
        \vol^\psi_n(L) - \vol^\psi_n(K)
        = \int _{\partial K}  \frac{\bx \cdot \bn_K(\bx)}{\|\bx\|^n} \left(\int_{\|\bx\|}^{\|\bx^L\|} \psi\left(t\frac{\bx}{\|\bx\|}\right)t^{n-1}\dd t\right) \, \dd\bx,
    \]
    where $\bx^L=\{r\bx: r\geq 0\} \cap \partial L$.
\end{lemma}
\begin{proof}
    The points $\by\in L\setminus (\interior K)$ can be parametrized by $\by(\bx,t) = t\frac{\bx}{\|\bx\|}$ for $\bx\in\partial K$ and $\|\bx\|\leq t \leq \|\bx^L\|$.
    The map $\by$ is bilipschitz and for almost all $\bx\in\partial K$ the Jacobian is $J\by(\bx,t)=(t/\|\bx\|)^{n-1} (\frac{\bx}{\|\bx\|}\cdot \bn_K(\bx))$.
    This completes the proof.
\end{proof}

Thus we can write
\begin{align}
    &\lim_{\delta \to 0^+} \frac{\vol^\psi_n\left(\cI_\delta^\varphi(K)\right) - \vol^\psi_n(K)}{\delta^\frac{2}{n+1}}
        \notag\\
    &\qquad= \lim_{\delta \to 0^+} \int_{\partial K}
        \frac{\bx \cdot \bn_K(\bx)}{\delta^\frac{2}{n+1} \|\bx\|^n} \left(\int_{\|\bx\|}^{\|\bx^\delta\|} \psi\left(t\frac{\bx}{\|\bx\|}\right)t^{n-1}\dd t\right)\, \dd\bx
        \label{eqn:step1}
\end{align}
where $\bx^\delta=\{r \bx: r \geq 0\} \cap \partial \cI_\delta^\varphi(K)$.

Next, we want to interchange integration and limit on the right hand side. To do so,  we recall the \emph{rolling function} $r_K:\partial K \to [0,\infty)$ of a convex body $K$, which  was introduced by McMullen in \cite{McMullen:1974}, see also \cite{SW:1990}.
For $\bx \in \partial K$ with unique outer normal $\bn_K(\bx)$ it is defined by
\begin{equation*}
	r_K(\bx) = \max\{r: B^{n}_2 (\bx - r\,  \bn_{K} (\bx), r ) \subset K\},
\end{equation*}
i.e., $r_K(\bx)$ is the maximal radius of a Euclidean ball inside $K$ that contains $\bx$.
If $\bn_K(\bx)$ is not unique, then we set $r_K(\bx) =0$. By McMullen \cite{McMullen:1974} (also \cite{SW:1990}),  $r_K(\bx)>0$  almost everywhere on $\partial K$.
It was shown in \cite{SW:1990} that for all $0 \leq \alpha <1$,
\begin{equation}\label{eqn:rolling_function}
    \int_{\partial K}r_K(\bx)^{-\alpha} \,\dd\bx < \infty.
\end{equation}

Our next lemma lets us apply this bound in our situation.
\begin{lemma}\label{lem:bounded}
There exists $\delta_0>0$ and $C>0$ such that for all $\bx \in \partial K$ and $\delta\in(0,\delta_0)$ we have
\begin{equation*}
    \frac{\bx \cdot \bn_K(\bx)}{\delta^\frac{2}{n+1} \|\bx\|^n} \left(\int_{\|\bx\|}^{\|\bx^\delta\|} \psi\left(t\frac{\bx}{\|\bx\|}\right)t^{n-1}\dd t\right)
    \leq C r_K(\bx)^{-\frac{n-1}{n+1}}.
\end{equation*}
\end{lemma}

\begin{proof}
Since $\cI_\delta^\varphi(K)$ converges to $K$ in the Hausdorff metric as $\delta\to 0^+$ there exists $\delta_1>0$ such that $L:=\cI_{\delta_1}^\varphi(K)\subset U$ is a compact star body. Since $\varphi:U\to (0,+\infty)$ is continuous, there exist constants $0<m_{\varphi}\leq M_\varphi$ such that
\begin{equation*}
    m_{\varphi}\leq \varphi(\by) \leq M_\varphi\qquad \text{for all $\by \in L$},
\end{equation*}
and similarly there exist constants $0<m_\psi\leq M_\psi$ such that
\begin{equation*}
    m_\psi \leq \psi(\by) \leq M_\psi \qquad \text{for all $\by \in L$}.
\end{equation*}
Thus
\begin{equation*}
    m_\varphi V_K(\bz) \leq V_K^\varphi(\bz) =\vol_n^\varphi([\bz,K]\setminus K) \leq M_\varphi V_K(\bz),
\end{equation*}
for all $\bz \in L\setminus K$. The sublevel sets therefore yield
\begin{equation}\label{eqn:compare_uniform}
    \cI_{\delta/M_\varphi}(K)\subset \cI_{\delta}^\varphi(K) \subset \cI_{\delta/m_\varphi}(K),
\end{equation}
for $\delta \in (0,\delta_1)$.

We set $\tilde{\delta} =\frac{\delta}{m_\varphi}$ and $\tilde{\bx}^{\tilde{\delta}} = \{r\bx:r\geq 0\}\cap \partial \cI_{\tilde{\delta}}(K)$. Note that $\|\tilde{\bx}^{\tilde{\delta}}\|\geq \|\bx^\delta\|$ by \eqref{eqn:compare_uniform}.

Thus
\begin{align*}
    \frac{\bx \cdot \bn_K(\bx)}{\delta^\frac{2}{n+1} \|\bx\|^n} \left(\int_{\|\bx\|}^{\|\bx^\delta\|} \psi\left(t\frac{\bx}{\|\bx\|}\right)t^{n-1}\dd t\right)
    &\leq M_\psi \frac{\bx \cdot \bn_K(\bx)}{\delta^\frac{2}{n+1} \|\bx\|^n} \left(\int_{\|\bx\|}^{\|\tilde{\bx}^{\tilde{\delta}}\|}t^{n-1}\dd t\right)\\
    &= \frac{M_\psi \, (\bx \cdot \bn_K(\bx))}{n\, (m_\varphi \tilde{\delta})^\frac{2}{n+1}}
        \left[\left(\frac{\|\tilde{\bx}^{\tilde{\delta}}\|}{\|\bx\|}\right)^{\!\!n}-1\right]
\end{align*}
By Lemma 2 in \cite{Werner:1994} there exists $\gamma>0$ and $\delta_2>0$ such that
\begin{equation*}
    \frac{\bx\cdot\bn_K(\bx)}{n\, \tilde{\delta}^\frac{2}{n+1}} \left[\left(\frac{\|\tilde{\bx}^{\tilde{\delta}}\|}{\|\bx\|}\right)^n-1\right]\leq \gamma r_K(\bx)^{-\frac{n-1}{n+1}},
\end{equation*}
for all $\bx\in\partial K$ and $\tilde{\delta} \leq \delta_2$.
Setting $C= M_\psi m_\varphi^{-\frac{2}{n+1}} \gamma$ and $\delta_0 = \min\{\delta_1, m_\varphi\delta_2\}$ completes the proof.
\end{proof}

By the previous lemma and (\ref{eqn:rolling_function}) we may apply Lebesgue's Dominated Convergence Theorem to interchange integration and limit in \eqref{eqn:step1} and obtain
\begin{align*}
    &\lim_{\delta \to 0^+} \frac{\vol^\psi_n\left(\cI_\delta^\varphi(K)\right) - \vol^\psi_n(K)}{\delta^\frac{2}{n+1}}\\
    &\qquad =  \int_{\partial K} \lim_{\delta \to 0^+} \frac{\bx \cdot \bn_K(\bx)}{\delta^\frac{2}{n+1} \|\bx\|^n}
        \left(\int_{\|\bx\|}^{\|\bx^\delta\|} \psi\left(t\frac{\bx}{\|\bx\|}\right) t^{n-1}\, \dd t \right) \,\dd \bx.
\end{align*}
Since $\bx^\delta \to \bx$ as $\delta \to 0^+$ we find
\begin{equation*}
    \lim_{\delta\to 0^+ } \frac{1}{\|\bx^\delta\|-\|\bx\|} \int_{\|\bx\|}^{\|\bx^\delta\|} \psi\left(t\frac{\bx}{\|\bx\|}\right) t^{n-1}\, \dd t = \psi(\bx) \|\bx\|^{n-1}.
\end{equation*}
Furthermore, since $\bx,\bx^\delta$ and $\bo$ are colinear, we have
\begin{equation*}
    \frac{\bx\cdot\bn_K}{\|\bx\|} (\|\bx^\delta\|-\|\bx\|) = \frac{(\bx^\delta-\bx)\cdot\bn_K}{\|\bx^\delta-\bx\|} \|\bx^\delta-\bx\| = (\bx^\delta-\bx)\cdot\bn_K.
\end{equation*}
Thus
\begin{equation}\label{eqn:step2}
    \lim_{\delta \to 0^+} \frac{\vol^\psi_n\left(\cI_\delta^\varphi(K)\right) - \vol^\psi_n(K)}{\delta^\frac{2}{n+1}}\\
     = \int_{\partial K} \lim_{\delta\to 0^+} \frac{(\bx^\delta-\bx)\cdot \bn_K(\bx)}{\delta^{\frac{2}{n+1}}} \,\psi(x) \, \dd\bx.
\end{equation}

Next we treat the limit under the integral.

\begin{lemma}\label{lem:limit}
For almost all $\bx\in\partial K$ the limit
\begin{equation*}
    \lim_{\delta \to 0^+}  \frac{(\bx^\delta-\bx)\cdot\bn_K(\bx)}{\delta^{\frac{2}{n+1}}}
\end{equation*}
exists and is equal to
\begin{equation*}
    c_n\, H_{n-1}(K,\bx)^\frac{1}{n+1} \, \varphi(\bx)^{-\frac{2}{n+1}},
\end{equation*}
where $c_n$ is defined as in \eqref{def:cn}.
\end{lemma}

\begin{proof}
Let $\varepsilon \in (0,1)$ and let $\bx \in \partial K$ be arbitrary.
Since $\varphi$ is continuous, there exists a neighborhood $B_{\eta}(\bx)$ of $\bx$ such that for all $ \by \in B_{\eta}(\bx)$  we have
\[
    (1-\varepsilon)\, \varphi(\bx) \leq \varphi(\by) \leq (1+\varepsilon) \, \varphi(\bx).
\]
We first assume that $\bx\in\partial K$ is a normal boundary point with $H_{n-1}(K,\bx)>0$. Then we may choose $\delta_0$ so small that for all
$\delta < \delta_0$ we have that
$$[\bx^\delta,K]\setminus K \subset B_{\eta}(\bx).$$
Recalling that $V_K(\bz) = \vol_n([\bz,K]\setminus K)$, we find
\begin{equation*}
    (1-\varepsilon) \, \varphi(\bx) \, V_K(\bx^\delta)\leq \delta = V_K^{\varphi}(\bx^\delta) \leq (1+\varepsilon)\,  \varphi(\bx) \, V_K(\bx^\delta),
\end{equation*}
or, equivalently
\begin{equation*}
    \underline{\delta} := \frac{\delta}{(1+\varepsilon)\varphi(\bx)}
    \leq  V_K(\bx^\delta)
    \leq \frac{\delta}{(1-\varepsilon)\varphi(\bx)} =
    :\overline{\delta}.
\end{equation*}
We set
\begin{equation*}
    \underline{\bx}^{\underline{\delta}} = \{r\bx : r\geq 0\} \cap \cI_{\underline{\delta}}(K), \quad
    \overline{\bx}^{\overline{\delta}} = \{r\bx : r\geq 0\} \cap \cI_{\overline{\delta}}(K)
\end{equation*}
and note that $\bo,\bx,\underline{\bx}^{\underline{\delta}}, \bx^\delta$, and $\overline{\bx}^{\overline{\delta}}$ are all colinear and in this order on the half-line $\{r\bx:r\geq 0\}$.

Thus
\begin{align*}
    &\limsup_{\delta\to 0^+} \frac{(\bx^\delta-\bx)\cdot\bn_K(\bx)}{\delta^{\frac{2}{n+1}}}\\
    &\quad = \limsup_{\delta\to 0^+} \frac{\bx\cdot\bn_K(\bx)}{\delta^{\frac{2}{n+1}}} \left(\frac{\|\bx^\delta\|}{\|\bx\|}-1\right)\\
    &\quad \leq \limsup_{\delta\to 0^+} \frac{\bx\cdot\bn_K(\bx)}{\delta^{\frac{2}{n+1}}} \left(\frac{\|\overline{\bx}^{\overline{\delta}}\|}{\|\bx\|}-1\right)\\
    &\quad = ((1-\varepsilon)\varphi(\bx))^{-\frac{2}{n+1}} \limsup_{\overline{\delta}\to 0^+} \frac{\bx\cdot\bn_K(\bx)}{\overline{\delta}^{\frac{2}{n+1}}} \left(\frac{\|\overline{\bx}^{\overline{\delta}}\|}{\|\bx\|}-1\right)\\
    &\quad = c_n H_{n-1}(K,\bx)^{\frac{1}{n+1}} ((1-\varepsilon)\varphi(\bx))^{-\frac{2}{n+1}},
\end{align*}
where in the last step we used the limit theorem for uniform weights \cite[Lem.\ 3]{Werner:1994}.
Similarly we find
\begin{align*}
    \liminf_{\delta\to 0^+} \frac{(\bx^\delta-\bx)\cdot\bn_K(\bx)}{\delta^{\frac{2}{n+1}}}
    &\geq ((1+\varepsilon)\varphi(\bx))^{-\frac{2}{n+1}} \liminf_{\underline{\delta}\to 0^+} \frac{\bx\cdot\bn_K(\bx)}{\underline{\delta}^{\frac{2}{n+1}}} \left(\frac{\|\underline{\bx}^{\underline{\delta}}\|}{\|\bx\|}-1\right)\\
    &= c_n H_{n-1}(K,\bx)^{\frac{1}{n+1}} ((1+\varepsilon)\varphi(\bx))^{-\frac{2}{n+1}}.
\end{align*}

Finally, assume that $\bx\in\partial K$ is a normal boundary point such that $H_{n-1}(K,\bx)=0$. There is $\delta_1>0$ such that $L:=\cI_{\delta_1}^\varphi(K)$ is a compact star-body. Set
\begin{equation*}
    m_\varphi := \min_{\bx\in L} \varphi(\bx),
\end{equation*}
and note that $m_\varphi>0$ since $\varphi$ is continuous. Then
\begin{equation*}
    \delta =  \int_{[\bx^\delta,K]\setminus K} \varphi(\by) \,\dd\by \geq m_\varphi \vol_n([\bx^\delta,K]\setminus K),
\end{equation*}
for all $\delta\leq \delta_1$. Thus
\begin{equation*}
    V_K(\bx^\delta)= \vol_n([\bx^\delta,K]\setminus K) \leq \frac{\delta}{m_\varphi} =: \tilde{\delta},
\end{equation*}
and for $\tilde{\bx}^{\tilde{\delta}} = \{r\bx:r\geq 0\} \cap \cI_{\tilde{\delta}}(K)$ we have that $\bo,\bx,\bx^\delta$, and $\tilde{\bx}^{\tilde{\delta}}$ are colinear.
Thus
\begin{align*}
    0&\leq \limsup_{\delta\to 0^+} \frac{(\bx^\delta-\bx)\cdot\bn_K(\bx)}{\delta^{\frac{2}{n+1}}} \\
    &\leq m_\varphi^{-\frac{2}{n+1}} \limsup_{\tilde{\delta}\to 0^+} \frac{\bx\cdot\bn_K(\bx)}{\tilde{\delta}^{\frac{2}{n+1}}} \left(\frac{\|\tilde{\bx}^{\tilde{\delta}}\|}{\|\bx\|}-1\right)\\
    &=c_n H_{n-1}(K,\bx)^{\frac{1}{n+1}} m_{\varphi}^{-\frac{2}{n+1}} = 0.
\end{align*}
Since $\varepsilon >0$ was arbitrary and almost all $\bx\in\partial K$ are normal, this completes the proof.
\end{proof}

We are now ready to complete the proof of our main theorem.

\begin{proof}[Proof of Theorem~\ref{thm:main_weighted}]
    From Lemma~\ref{lem:voldiff} and Lemma~\ref{lem:bounded} we conclude \eqref{eqn:step2} and apply Lemma~\ref{lem:limit}, that is,
    \begin{align*}
    \lim_{\delta \to 0^+} \frac{\vol^\psi_n\left(\cI_\delta^\varphi(K)\right) - \vol^\psi_n(K)}{\delta^\frac{2}{n+1}}
    &=  \int_{\partial K} \left(\lim_{\delta \to 0^+}  \frac{(\bx^\delta-\bx)\cdot\bn_K(\bx)}{\delta^{\frac{2}{n+1}}}\right) \psi(\bx)\, \dd\bx\\
    &= c_n \int_{\partial K} H_{n-1}(K,\bx)^{\frac{1}{n+1}} \varphi(\bx)^{-\frac{2}{n+1}} \psi(\bx)\, \dd\bx.
    \end{align*}
    This completes the proof.
\end{proof}

\section{Proofs of additional results}\label{sec:applications}

We now prove our results for projective geometries and dual volumes.

\subsection{Riemannian spaces of constant curvature}\label{sec:Riemannian}

\begin{proof}[Proof of Theorem~\ref{thm:main_riemannian}]
    For $K\in \cK_0(\Sp^n(\lambda))$ we choose a projective model $(\mathbb{B}^n_\lambda,g^\lambda)$ and identify $K$ with a Euclidean convex body $\overline{K}\subset \mathbb{B}^n_\lambda$. Furthermore, we may assume w.l.o.g.\ that $\bo\in\interior K$. By \eqref{eqn:space_form_illumination_is_weighted} we find
    \begin{equation*}
        \vol_n^\lambda(\cI^\lambda_\delta(K)) = \vol_n^{\varphi_\lambda}(\cI^{\varphi_{\lambda}}_\delta(\overline{K})).
    \end{equation*}
    By Theorem~\ref{thm:main_weighted} we conclude
    \begin{equation*}
        \lim_{\delta\to 0^+} \frac{\vol_n^\lambda(\cI_\delta^\lambda(K))-\vol_n^\lambda(K)}{\delta^{\frac{2}{n+1}}}
        = c_n\int_{\partial \overline{K}} H_{n-1}^e(\overline{K},\bx)^{\frac{1}{n+1}} \varphi_\lambda(\bx)^{\frac{n-1}{n+1}} \,\dd\bx.
    \end{equation*}
    Now notice that by \eqref{eqn:lgauss} we have
    \begin{equation*}
        H_{n-1}^e(\overline{K},\bx)^{\frac{1}{n+1}} \varphi_\lambda(\bx)^{\frac{n-1}{n+1}}
        = H_{n-1}^\lambda(\overline{K},\bx)^{\frac{1}{n+1}} \sqrt{\frac{1+\lambda(\bn_{\overline{K}}^e(\bx)\cdot \bx)^2}{(1+\lambda\|\bx\|^2)^n}},
    \end{equation*}
    and this together with \eqref{eqn:lboundary} yields
    \begin{align*}
        \int_{\partial \overline{K}} H_{n-1}^e(\overline{K},\bx)^{\frac{1}{n+1}} \varphi_\lambda(\bx)^{\frac{n-1}{n+1}} \,\dd\bx
        &=\!\! c_n\int_{\partial \overline{K}} H_{n-1}^\lambda(\overline{K},\bx)^{\frac{1}{n+1}} \, \vol_{\partial \overline{K}}^{\lambda}(\dd\bx)
        \\
        & = \!\! c_n\,\Omega^\lambda(K).
    \end{align*}
    This completes the proof.
\end{proof}

Next we aim to prove Theorem~\ref{thm:convex_hyperbolic_plane}, i.e., that the illumination body in the hyperbolic plane of a convex body is always convex. For this we first show the following lemmas.

\begin{lemma}\label{lem:hyperbolic_plane_triangle}
    Let $\bp,\bq\in\HH^2$, and let $\gamma:\RR\to \HH^2$ be a constant speed geodesic. Then the function $A:\RR\to [0,+\infty)$
    \begin{equation*}
        A(t):= \vol_2^h(\conv\{\gamma(t),\bp,\bq\}),
    \end{equation*}
    has no local maxima unless $A\equiv 0$.
\end{lemma}
\begin{proof}
    We work in the upper half plane model of $\HH^2$. We may, after applying an orientation-preserving isometry and a monotone reparametrization of $\gamma$, assume that
    \begin{align*}
        \gamma(t) &= (0,e^t),&
        \bp &= (x_1,y_1), & \text{and}&&
        \bq &= (x_0,1),
    \end{align*}
    with $y_1>0$.
    Set $\triangle(t) := \mathrm{conv}\{\gamma(t),\bp,\bq\}$. For the $1$-form $\omega := \frac{d x}{y}$ we have
    \[
        d\omega = \frac{dx\wedge dy}{y^2}.
    \]
    Thus, by Stokes' theorem,
    \begin{align*}
        A(t) = \vol_2^h(\triangle(t))
        = \int_{\triangle(t)} \frac{\dd(x,y)}{y^2}
        = \left|\int_{\partial \triangle(t)} \frac{\dd x}{y}\right|,
    \end{align*}
    It is an exercise to show that, if $\eta$ is a geodesic segment joining $(u_0,v_0)$  to $(u_1,v_1)$ on the upper half-plane, then
    \[
        \tan\left(\frac{1}{2} \int_{\eta} \frac{\dd x}{y}\right) = \frac{u_1-u_0}{v_0+v_1}.
    \]
    Applying this to the three sides of the triangle $\triangle(t)$, defining $\tau:(0,\infty)\to \RR$ by
    \[
        \tau(s) = \tan\left(\frac{1}{2} \int_{\partial \triangle(\ln s)} \frac{\dd x}{y}\right), s>0,
    \]
    and using the identity
    \[
        \tan(a+b+c) = \frac{\tan a + \tan b + \tan c - \tan a \tan b \tan c}{1-\tan a\tan b-\tan a\tan c-\tan b\tan c}
    \]
    we find that $\tau(s) = \frac{\alpha(s)}{\beta(s)}$ where
    \begin{align*}
        \alpha(s) &= x_1(x_0^2+1)-x_0(x_1^2+y_1^2)+(x_0-x_1)s^2,\\
        \beta(s) &= y_1(x_0^2+1)+x_1^2+y_1^2+((x_0-x_1)^2+(1+y_1)^2)s + (1+y_1)s^2.
    \end{align*}
    Differentiating, we obtain
    \[
        \tau'(s) = \frac{((x_0-x_1)^2+(1+y_1)^2)\mu(s)}{\beta(s)^2},
    \]
    where
    \[
        \mu(s) = (x_0-x_1)s^2+2(x_0y_1-x_1)s-x_1(x_0^2+1)+x_0(x_1^2+y_1^2).
    \]
    If $\tau(s)=0$, then $A(\ln s)=0$, so $A$ does not attain a local maximum at $t=\ln s$ unless it vanishes in a neighborhood of $t$, in which case $\gamma$ passes through $\bp$ and $\bq$ and we find $A\equiv 0$. Also, if $\tau(s)\neq 0$, then $|\tau|$ is differentiable at $s$ and so does not attain a maximum unless $\tau'(s)=0$.

    So assume that $s_0\in (0,+\infty)$ satisfies $\tau(s_0)\neq 0$ and $\tau'(s_0)=0$. Then $\mu(s_0)=0$, that is,
    \begin{equation*}
        x_1(x_0^2+1)-x_0(x_1^2+y_1^2) = (x_0-x_1)s_0^2 + 2(x_0y_1-x_1)s_0.
    \end{equation*}
    This yields
    \begin{align*}
        \alpha(s_0) &= 2(x_0-x_1) s_0^2+2(x_0y_1-x_1)s_0\\
        \intertext{and}
        \mu'(s_0)&=2(x_0-x_1)s_0+2(x_0y_1-x_1).
    \end{align*}
    Differentiating again we find that
    \begin{equation*}
        \tau''(s_0) = \frac{((x_0-x_1)^2+(1+y_1)^2)\alpha(s_0)}{s_0\beta(s_0)^2} = \frac{(x_0-x_1)^2+(1+y_1)^2}{s_0\beta(s_0)}\tau(s_0).
    \end{equation*}
    In particular
    \[
        \operatorname{sgn} \tau''(s_0) = \operatorname{sgn} \tau(s_0),
    \]
    whence the function $|\tau|$ does not have a local maximum at $s_0$. Since $A(t)=2\arctan(|\tau(e^t)|)$, it follows that the function $A$ does not have a local maximum at any $t=\ln s\in\RR$ unless $A\equiv 0$.
\end{proof}

\begin{figure}[t]
\centering
\begin{tikzpicture}[scale=0.6]

\fill[black,opacity=0.05]
    (2,0) rectangle (10,12);

\fill[blue, opacity =0.3]
  (0,1)
  .. controls (1,0)  and (1.5,0.5)  .. (2,1)
  .. controls (2.5,1.5) and (2.5,1.5) .. (3,2)
  .. controls (4,3)  and (4,6)  .. (2,8)
  .. controls (0,10)  and (-2,10)  .. (-3,7)
  .. controls (-4,4) and (-2,3) .. (0,1);

\fill [orange, opacity=0.3]
    (2,1) -- (5.5,4.5) -- (2,8)--cycle;

\draw[thick, blue]
  (0,1)
  .. controls (1,0)  and (1.5,0.5)  .. (2,1)
  .. controls (2.5,1.5) and (2.5,1.5) .. (3,2)
  .. controls (4,3)  and (4,6)  .. (2,8)
  .. controls (0,10)  and (-2,10)  .. (-3,7)
  .. controls (-4,4) and (-2,3) .. (0,1);

\draw (2,1) -- (5.5,4.5) -- (2,8);
\draw[thick] (4.25,0.75) -- (5.5,4.5) -- (7.5,10.5);
\draw[dashed] (2,1) -- (7,9) -- (2,8);
\draw (2,0) -- (2,12);

\draw[dotted] (-0.4,9.4) -- (7,9) -- (3.35,2.5);

\fill (2,8) circle (3pt) node[below left] {$\bp$};
\fill (2,1) circle (3pt) node[below right] {$\bq$};

\fill (5.5,4.5) circle (3pt) node[right] {$\gamma(t_0)$};
\fill (7,9) circle (3pt) node[below right] {$\gamma(t_i)$};

\node at (8,1) {$H^+$};
\node at (-2,5) {$K$};

\node[right] at (4.5,1.5) {$\gamma$};

\end{tikzpicture}
\caption{Sketch for the proof of Proposition~\ref{prop:areanomax}.
Given a geodesic segment $\gamma$ that is outside of $K$ we consider the point $\gamma(t_0)$ that determines the boundary points $\bp$ and $\bq$. $H^+$ is the closed half-space that is determined by the points $\bp,\bq\in\partial K$ and for $|t-t_0|$ small enough $\gamma(t)$ is contained in $H^+$.
Then the convex hull $[\gamma(0),K]$ is the union of $K$ (in blue) with the triangle $\triangle(t_0)=\conv\{\bp,\bq,\gamma(t_0)\}$ (in orange). The triangle $\triangle(t_1)=\conv\{\bp,\bq,\gamma(t_i)\}$ (dashed lines) is contained in the convex hull $[\gamma(t_i),K]$ (dotted lines).}
\label{fig:KHpq}
\end{figure}
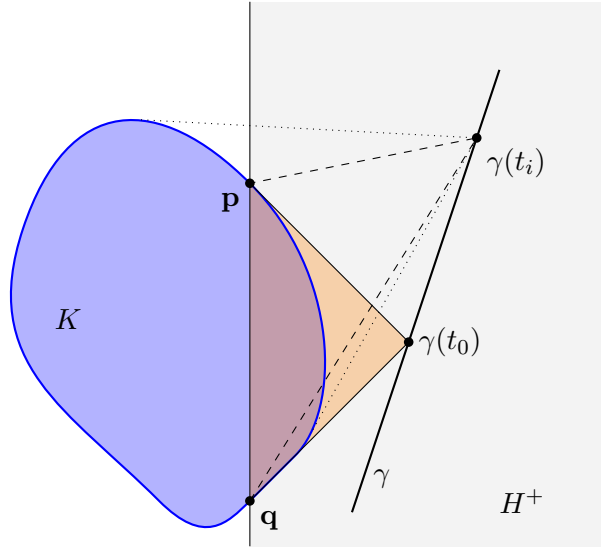

\begin{proposition}\label{prop:areanomax}
    Let $K\subset \HH^2$ be a convex body and let $\gamma:[0,1]\to \HH^2$ be a constant speed geodesic segment lying outside of $K$. Then the function
    $B:(0,1)\to [0,+\infty)$ defined by
    $B(t):=\vol_2^h([\gamma(t),K])$ has no local maximum.
\end{proposition}
\begin{proof}
    Let $t_0\in (0,1)$ be arbitrary. Let $\bp,\bq\in \partial K$ be the contact points of supporting geodesics joining $\gamma(t_0)$ to $K$; see Figure~\ref{fig:KHpq}. Set $\triangle(t) = \conv\{\gamma(t),\bp,\bq\}$ for $t\in(0,1)$ and note that
    \begin{equation*}
        [\gamma(t_0),K] = K\cup \triangle(t_0).
    \end{equation*}
    Thus
    \begin{equation}\label{eqn:proof_step_1}
        B(t_0) = \vol_2^h([\gamma(t_0),K]) = \vol_2^h(K) + \vol_2(\triangle(t_0)) - \vol_2^h(K\cap \triangle(t_0)).
    \end{equation}
    Moreover,
    \begin{equation}\label{eqn:proof_step_2}
        K\cap \triangle(t_0) = K\cap H^+,
    \end{equation}
    where $H^+$ is the closed half plane containing $\bp$ and $\bq$ on its boundary and $\gamma(t_0)$ in its interior. By Lemma~\ref{lem:hyperbolic_plane_triangle}, the function $A(t)=\vol_2^h(\triangle(t))$ does not attain a local maximum at $t_0$. Hence there exists a sequence $(t_i)_{i\geq 1}$ such that $\lim_{i\to\infty} t_i = t_0$ and
    \begin{equation}\label{eqn:proof_step_3}
        A(t_i)=\vol_2^h(\triangle(t_i)) > \vol_2^h(\triangle(t_0))=A(t_0) \qquad \text{for all $i\geq 1$}.
    \end{equation}
    Since $\gamma(t_0)$ lies in the interior of $H^+$, we may assume that $\gamma(t_i)\in H^+$ for all $i\geq 1$. Hence, by \eqref{eqn:proof_step_2},
    \begin{equation}\label{eqn:proof_step_4}
        K\cap\triangle(t_i) \subset K\cap H^+ = K\cap\triangle(t_0)\qquad \text{for all $i\geq 1$.}
    \end{equation}
    For every $i\geq 1$ the set $[\gamma(t_i),K]$ contains the set $K$ and the triangle $\triangle(t_i)=\conv\{\gamma(t_i),\bp,\bq\}$. Hence, by \eqref{eqn:proof_step_3}, \eqref{eqn:proof_step_4}, and \eqref{eqn:proof_step_1},
    \begin{align*}
        B(t_i) = \vol_2^h([\gamma(t_i),K])
            &\geq \vol_2^h(K\cup \triangle(t_i))\\
            &= \vol_2^h(K) + A(t_i) - \vol_2(K\cap \triangle(t_i))\\
            &> \vol_2^h(K) + A(t_0) - \vol_2(K\cap \triangle(t_0)) = B(t_0).
    \end{align*}
    It follows that $B$ has no local maximum at $t_0$. This completes the proof since $t_0\in(0,1)$ was arbitrary.
\end{proof}

\begin{proof}[Proof of Theorem~\ref{thm:convex_hyperbolic_plane}]
    Let $K\subset \HH^2$ be a convex body. The illumination body $\cI_\delta^h(K)$ is a sublevel set of the function
    \begin{equation*}
        \bp\mapsto \vol_2^h([\bp,K])
    \end{equation*}
    on the hyperbolic plane. A continuous function $f$ on the hyperbolic plane has convex sublevel sets if and only if, for every geodesic $\gamma$, the function $f\circ \gamma$ has no interior local maxima. Hence the theorem follows from Proposition~\ref{prop:areanomax}.
\end{proof}

Finally we turn our attention to the spherical plane.

\begin{lemma}\label{lem:spherical_triangle}
Let $\bp,\bq\in\SS^2$ be distinct non-antipodal points. Then for every $\delta\in(0,\pi)$ there exists a geodesic segment
\[
\gamma:(-\varepsilon,\varepsilon)\to \SS^2
\]
such that
\[
A(t):=\vol_2^s(\mathrm{conv}\{\gamma(t),\bp,\bq\})
\]
has a strict local maximum at $t=0$ and satisfies $A(0)=\delta$.
\end{lemma}

\begin{proof}
Let $\bp,\bq\in\SS^2$ be fixed and set
\[
\varphi := \tfrac12 d(\bp,\bq)\in(0,\tfrac{\pi}{2}).
\]
By an isometry of $\SS^2$, we may assume that
\[
\bp=(\cos\vartheta\cos\varphi, \sin\varphi, \sin\vartheta\cos\varphi),
\qquad
\bq=(\cos\vartheta\cos\varphi, -\sin\varphi, \sin\vartheta\cos\varphi),
\]
for any $\vartheta\in(0,\pi-\varphi)$.

Let
\[
\gamma(t)=(\cos t,\sin t,0),
\]
so that $\gamma$ is a great circle and $\bp,\bq$ lie on a rotated great circle obtained from $\gamma$ by rotation through the angle $\vartheta$ about the $y$-axis.

For a spherical triangle with vertices $A,B,C\subset\SS^2\subset \RR^3$, we use the solid angle formula for the spherical excess $E$
\[
\tan\frac{E}{2}
=
\frac{|\det(A,B,C)|}
{1+A\cdot B+B\cdot C+C\cdot A},
\]
see, for example, \cite{vOS:1983}.

Applied to the triangle $\conv\{\gamma(t),\bp,\bq\}\subset \SS^2$, a direct computation gives
\[
\det(\gamma(t),\bp,\bq)
=
2\sin\varphi\,\sin\vartheta\,\cos\varphi\,\cos t,
\]
and
\[
1+\gamma(t)\cdot\bp+\bp\cdot\bq+\bq\cdot\gamma(t)
=
2\cos\varphi\bigl(\cos\varphi+\cos\vartheta\cos t\bigr).
\]

Hence
\[
\tan\frac{A(t)}{2}
=
\frac{\sin\varphi\,\sin\vartheta\,\cos t}{\cos\varphi+\cos\vartheta\cos t}.
\]

The right-hand side is an even smooth function of $t$ since it depends on $\cos t$ that attains a strict local maximum at $t=0$. Since $x\mapsto \tan(x/2)$ is strictly increasing on $(0,\pi)$, it follows that $A(t)$ attains a strict local maximum at $t=0$.
Since
\[
A(0)
=
2\arctan\!\left(
\frac{\sin\varphi\,\sin\vartheta}{\cos\varphi+\cos\vartheta}
\right).
\]
we find that $(0,\pi-\varphi)\ni \vartheta\mapsto A(0)$ varies continuously from $0$ to $\pi$, and therefore attains every value in $(0,\pi)$.
Thus for every $\delta\in(0,\pi)$ there exists $\vartheta$ such that $A(0)=\delta$, and $\gamma$ yields the claim.
\end{proof}

\begin{theorem}\label{thm:spherical_not_convex}
    If the boundary of a convex body $K\subset \SS^2$ contains a geodesic segment, then $\cI_\delta^s(K)$ is not convex for all sufficiently small $\delta>0$.
\end{theorem}
\begin{proof}
Suppose that $\partial K$ contains a geodesic segment with endpoints $\bp$ and $\bq$.
The great circle through $\bp$ and $\bq$ determines a supporting hemisphere $H^+$ whose interior is disjoint from $K$.

If $\bx\in \interior H^+$ is sufficiently close to the relative interior of $[\bp,\bq]$, then
\[
\vol_2^s([\bx,K]) = \vol_2^s(K) + \vol_2^s(\conv\{\bx,\bp,\bq\}).
\]

By Lemma~\ref{lem:spherical_triangle}, for every sufficiently small $\delta>0$ there exists a geodesic $\gamma=\gamma_\delta \subset \interior H^+$ such that
\[
    A(t):=\vol_2^s(\conv\{\gamma(t),\bp,\bq\})
\]
has a strict local maximum at $t=0$ and satisfies $A(0)=\delta$ as well as
\[
    A(t) = \vol_2^s([\gamma(t),K])-\vol_2^s(K),
\]
for sufficiently small $|t|>0$.
Since $A(t)$ has a strict local maximum at $t=0$, there exist $t_-<0<t_+$ such that
\[
    A(t_\pm) < A(0).
\]

Choose $\delta' \in ( \max\{A(t_-),A(t_+)\}, A(0) )$.
Then
\[
    \vol_2^s([\gamma(t_\pm),K]) - \vol_2^s(K) = A(t_\pm) < \delta',
\]
so $\gamma(t_\pm)\in \mathcal I_{\delta'}^s(K)$.
On the other hand,
\[
    \vol_2^s([\gamma(0),K]) - \vol_2^s(K) = A(0) > \delta',
\]
so $\gamma(0)\notin \cI_{\delta'}^s(K)$.
Hence $\cI_{\delta'}^s(K)$ is not convex.
\end{proof}

\subsection{Projective Finsler geometries} \label{sec:Finsler}

\begin{proof}[Proof of Theorem~\ref{thm:main_finsler}]
    Let $K\in\cK_0(X)$ and let $\vol_n^F$ be a definition of volume on $(X,F)$ with continuous density $\varphi_F:X\to (0,+\infty)$. Since $(X,F)$ is a projective Finsler geometry, $K\subset \RR^n$ is Euclidean convex body and we have
    \begin{equation*}
        \cI^F_\delta(K) = \cI^{\varphi_F}_\delta(K).
    \end{equation*}
    Furthermore, since $\cI^{\varphi_F}_\delta(K)$ converges to $K$ as $\delta\to 0^+$ there exists $\delta_0$ such that $L=\cI^{\varphi_F}_{\delta_0}(K)$ is a compact star-body. Thus $\varphi_F$ restricted to $\interior L$ is integrable and we conclude by Theorem~\ref{thm:main_weighted}
    \begin{equation*}
        \lim_{\delta\to 0^+} \frac{\vol_n^F(\cI^F_\delta(K))-\vol_n^F(K)}{\delta^{\frac{2}{n+1}}}
        = c_n \int_{\partial K} H_{n-1}^e(K,\bx)^{\frac{1}{n+1}} \varphi_F(\bx)^{\frac{n-1}{n+1}}\, \dd\bx.
    \end{equation*}
    This completes the proof.
\end{proof}

\subsection{Dual volumes}

\begin{proof}[Proof of Theorem~\ref{thm:main_dual}]
    For $q>0$ and a star body $L\subset \RR^n$ we find
    \begin{align*}
        \widetilde{V}_q(L)
        &= \frac{q}{n} \int_{\SS^{n-1}} \int_{0}^{\rho(L,\bu)} s^{q-1} \, \dd s \, \dd\bu
        = \frac{q}{n} \int_{L} \|\bx\|^{q-n}\, \dd\bx.
    \end{align*}
    where we used polar coordinates $\bx=s\bu$.

    Let $K\in\cK(\RR^n)$ be such that $\bo\in \interior K$. Then there exists $\rho>0$ such that $B_{\bo}(\rho)\subset \interior K$.
    Thus
    \begin{align*}
        \widetilde{V}_q(\cI_\delta(K))-\widetilde{V}_q(K)
            &= \frac{q}{n} \int_{\cI_\delta(K)\setminus K} \|\bx\|^{q-n}\,\dd\bx\\
            &= \vol_n^{\psi_q}(\cI_\delta(K)) -\vol_n^{\psi_q}(K),
    \end{align*}
    where we define the continuous function $\psi_q:\RR^n\to (0,\infty)$ for $q>0$ by
    \begin{equation*}
        \psi_q(\bx) = \begin{cases}
                            \frac{q}{n} \|\bx\|^{q-n} & \text{if $\|\bx\|\geq \rho$,}\\
                            \frac{q}{n} \rho^{q-n}  &\text{ if $\|\bx\|<\rho$}.
                       \end{cases}
    \end{equation*}

    Similarly, for $q<0$ we find
    \begin{align*}
        \widetilde{V}_q(L)
        &= \frac{-q}{n} \int_{\SS^{n-1}}\int_{\rho(L,\bu)}^{+\infty} s^{q-1}\,\dd s \, \dd\bu
        = \frac{|q|}{n} \int_{\RR^n\setminus L} \|\bx\|^{q-n}\, \dd\bx.
    \end{align*}
    Thus, in this case, we have
    \begin{align*}
        0\leq \widetilde{V}_q(K)-\widetilde{V}_q(\cI_\delta(K))
            &= \frac{|q|}{n} \int_{\cI_\delta(K)\setminus K} \|\bx\|^{q-n}\,\dd\bx\\
            &= \vol_n^{\psi_q}(\cI_\delta(K)) - \vol_n^{\psi_q}(K),
    \end{align*}
    where we now set for $q<0$
    \begin{equation*}
        \psi_q(\bx) = \begin{cases}
                            \frac{|q|}{n} \|\bx\|^{q-n} & \text{if $\|\bx\|\geq \rho$,}\\
                            \frac{|q|}{n} \rho^{q-n}  &\text{ if $\|\bx\|<\rho$}.
                       \end{cases}
    \end{equation*}
    Now the statement of the theorem follows by Theorem~\ref{thm:main_weighted} for $\psi=\psi_q$ and $\varphi\equiv 1$.
\end{proof}

\subsection{\texorpdfstring{$L_p$}{Lp}-affine surface areas}

The $L_p$-affine surface areas were  introduced by Lutwak \cite{Lutwak:1996}, see also \cite{SW:2004}. For a convex body $K$ in $\mathbb{R}^n$ and $p \in \mathbb{R}$, $p \neq -n$, they are defined as
\begin{equation*}
    \as_{p}(K) = \int_{\partial K} \left(\frac{H_{n-1}(K, \bx)} {(\bx \cdot \bn_K(\bx))^{n+1}}\right)^{\!\!\frac{p}{n+p}} (\bx \cdot \bn_K(\bx)) \, \dd\bx.
\end{equation*}
The case $p=1$ is the classical affine surface area, going back to Blaschke \cite{Blaschke:1923}.
 The $L_p$-affine surface areas are semi continuous,  linear invariant valuations \cite{Ludwig:2001, Lutwak:1996}, they satisfy affine isoperimetric inequalities \cite{Lutwak:1996,WY:2008} and  they have a Steiner formula associated with them \cite{TW:2019}.
Consequently these quantities are important in many applications, e.g.,
the approximation theory of convex bodies by polytopes
\cite{Boeroeczky:2000, Reitzner:2002, SW:2003},
affine curvature flows \cite{Andrews:1999, Ivaki:2015},
and information theory \cite{Artstein-Avidan:2012, CW:2014, CFGLSW:2014, PW:2012, Werner:2012}.
\vskip 2mm
We now briefly sketch how the $L_p$-affine surface areas can be obtained as a  corollary to Theorem~\ref {thm:main_weighted}.
\newline
Let $K$ be a convex body in $\mathbb{R}^n$ that is $C^2_+$.
We specify $\psi$ and  $\varphi$
to $\varphi_p$ and $\psi_p$ such that for $\bx \in \partial K$,
\[
\psi_p(\bx) \varphi_p(\bx)^{-\frac{2}{n+1}}
    = \frac{H_{n-1}(K, \bx)^{\frac{p}{n+p}-\frac{1}{n+1}}}{(\bx \cdot \bn_K(\bx))^{\frac{n(p-1)}{n+p}}}
\]
and extend $\varphi_p$ and $\psi_p$ continuously to $U$.
Then the  $L_p$ affine surface area is  a right-derivative at $\delta=0$:
\[
\lim_{\delta\to 0^+} \frac{\vol_n^{\psi_p}(\cI_\delta^{\varphi_p}(K))- \vol_n^{\psi_p}(K)}{\delta^\frac{2}{n+1}} = c_n\, \as_p(K).
\]
\appendix

\section{Geometric examples for illumination bodies}

In this section we write $\kappa_n = \vol_n(B_2^n)$ for the volume of the Euclidean unit ball and $\omega_n = \mathcal{H}^{n}(\SS^n)$ for the surface area of the unit ball $B_2^{n+1}$. Note that $\kappa_n = n\omega_n$ and $\omega_n=2\pi\kappa_{n-1}$.

\subsection{Weighted illumination body}

    While the classical illumination body is convex and bounded for all $\delta>0$, note that $\cI_\delta^\varphi(K)$ need not be convex or bounded as the following example shows. For example, if $\varphi$ has compact support, then $\cI_\delta^\varphi(K)$ will be unbounded once $\delta$ is large enough. We give a simple example in $\RR^2$.

    \begin{example}[weighted illumination body of a square]
        Let $U=(-2,2)^2$ and $K=[-1,1]^2$ and consider the weight $\varphi\equiv 1$ on $U$. For $\be_1=(1,0)^\top\in\SS^1$ we have that $\lim_{R\to\infty} \vol_2^{\varphi}([R\be_1,K]) = \vol_2([-1,2]\times[-1,1]) = 2$ and therefore $\rho(\cI^\varphi_\delta(K),\be_1)=+\infty$ for $\delta\geq 2$. For $\bu=(\frac{1}{\sqrt{2}},\frac{1}{\sqrt{2}})^\top\in\SS^1$ we have $\lim_{R\to\infty} \vol_2^\varphi([R\bu,K]) = 8$ and therefore $\rho(\cI^\varphi_\delta(K),\bu)<+\infty$ if $\delta< 4$. So for $\delta\in [2,4)$ the weighted illumination body $\cI_\delta^\varphi(K)$ is unbounded and not convex.
    \end{example}

\subsection{Spherical and hyperbolic illumination body}

By symmetry illumination bodies of geodesic balls in $\Sp^n(\lambda)$ are again geodesic balls.

\begin{example}[Illumination body of a geodesic ball in $\Sp^n(\lambda)$]
    Let $\lambda\in\RR$ and denote the geodesic distance between two points $\bx,\by\in\Sp^n(\lambda)$ by  $d(\bx,\by)$.
    The closed geodesic ball $B^\lambda_r(\bz):=\{\bx\in\Sp^n(\lambda) : d(\bx,\bz)\leq r\}$ of radius $r>0$ around $\bz$ is a convex body if
    \begin{enumerate}
        \item $\lambda\leq 0$, or
        \item if $\lambda >0$ and $r< \frac{\pi}{\sqrt{\lambda}}$.
    \end{enumerate}
    By symmetry around $\bz$ we find that the illumination body $\cI_\delta^\lambda(B_r(\bz))$ is also a closed geodesic ball $B_{r_\delta}(\bz)$ of radius $r_\delta$ determined by
    \begin{equation*}
        \delta = \vol_n([\exp_{\bz}(r_\delta U_{\bz}),B_r(\bz)]\setminus B_r(\bz)),
    \end{equation*}
    for some arbitrary unit direction $U_{\bz}\in T_{\bz} \Sp^n(\lambda)$ and $\lambda \leq 0$, or $\lambda >0$ and $\delta \leq \delta_r^\lambda := \frac{1}{2}\vol_n^\lambda (\Sp^n(\lambda)) - \vol_n^\lambda(B_r(\bz))$.

    For $\lambda >0$ and $\bx\in \Sp^n(\lambda)$ such that $V_{B_r^\lambda(\bz)}^\lambda(\bz)>\delta_r^\lambda$ we have that $[\bz,B_r^\lambda(\bz)]=\Sp^n(\lambda)$ and therefore
    \[
        V_{B_r^\lambda(\bz)}^\lambda(\bz) = D_r^\lambda := \vol_n^\lambda (\Sp^n(\lambda)) - \vol_n^\lambda(B_r(\bz))
    \]
    for all $\bz\in\Sp^n(\lambda)$ such that $V_{B_r^\lambda(\bz)}^\lambda(\bz)>\delta_r^\lambda$. Thus, in this case, we have $\cI_\delta^\lambda(B_r(\bz))=\cI_{\delta_r^\lambda}^\lambda(B_r(\bz))$ for all $\delta \in [\delta_r^\lambda,D_r^\lambda)$ and $\cI_\delta^\lambda(B_r(\bz)) = \Sp^n(\lambda)$ for all $\delta \geq D_r^\lambda$.
\end{example}

For $n\geq2$ and $\lambda>0$ one can show that for any polytope $P\in\mathcal{K}_0(\Sp^n(\lambda))$ the illumination body $I_\delta^\lambda(P)$ is not geodesically convex for all $\delta>0$. The ideas are similar to the following example.

\begin{example}[Illumination body of a spherical square]
    We consider $\RR^2$ as a geodesic model of an open half-space of $\Sp^2(1)$. Then any convex body $\overline{K}\in \mathcal{K}_0(\RR^2)$ determines a unique spherical convex body $K$ in the open half-space of $\Sp^2(1)\cong\SS^2$.
    Furthermore, the Riemannian volume measure on $\Sp^2(1)$ is induced by the weighted volume measure $\vol_2^{\varphi_s}$ with weight function $\varphi_s(x,y)= (1+x^2+y^2)^{-3/2}$, i.e., $\vol_2(K)=\vol_2^{\varphi_s}(\overline{K})$

    Consider a (spherical) square $S$ that is determined by $\overline{S}=[-1,1]^2$. Then $\cI_\delta(S) \cong \cI_\delta^{\varphi_s}(\overline{S})$ is not convex for all $\delta >0$. To see this consider a point $\bz = (r,0)^\top \in \RR^2\setminus \overline{S}$ for $r> 1$.
    Then $[\bz,\overline{S}]\setminus \overline{S}$ is the triangle $T(\bz) = [r\bu, (1,1)^\top, (1,-1)^\top]$ and we calculate
    \begin{equation*}
        V_{\overline{S}}^{\varphi_s}(\bz) = \vol_2^s(T(\bz)) = \int_{1}^r \int_{-\frac{r-s}{r-1}}^{\frac{r-s}{r-1}} (1+s^2+t^2)^{-\frac{3}{2}} \, \dd t\, \dd s.
    \end{equation*}
    For $\bz' = (r,1)^\top\in \RR^2\setminus \overline{S}$ for  $r> 1$ we have that $[\bz,\overline{S}]\setminus \overline{S}$ is the triangle $T(\bz')= [\bz',(1,1)^\top,(1,-1)^\top]$. Thus
    \begin{align*}
        V_{\overline{S}}^s(\bz') = \vol_2^s(T(\bz'))
        &= \int_{1}^{r} \int_{\frac{2s-(r+1)}{r-1}}^1 (1+s^2+t^2)^{-3/2}\,\dd t \, \dd s\\
        &= \int_{1}^{r} \int_{-\frac{r-s}{r-1}}^{\frac{r-s}{r-1}}
            \left(1+s^2 + \left(a+\frac{s-1}{r-1}\right)^2\right)^{-3/2} \, \dd a\,\dd s\\
        &< \int_{1}^{r} \int_{-\frac{r-s}{r-1}}^{\frac{r-s}{r-1}}
            \left(1+s^2 + a^2\right)^{-3/2} \, \dd a\,\dd s = V_{\overline{S}}^s(\bz).
    \end{align*}

    Thus for $\delta=V_{\overline{S}}^s(\bz')$ the points $\bz'=(r,1)^\top$ and $\bz''=(r,-1)^\top$ belong to $\cI_{\delta}^s(\overline{S}) =\{\bx : V_{\overline{S}}^s(\bx) \leq \delta\}$, but the midpoint $\bz=(r,0)^\top = \frac{\bz'+\bz''}{2}$ does not belong to $\cI_\delta^s(\overline{S})$. Hence $\cI_\delta^s(\overline{S})$ is not convex.
    Also see Figure~\ref{fig:spherical_square}
\end{example}

\begin{figure}
    \centering
    \begin{tikzpicture}
        \node at (0,0) {\includegraphics[width=6cm]{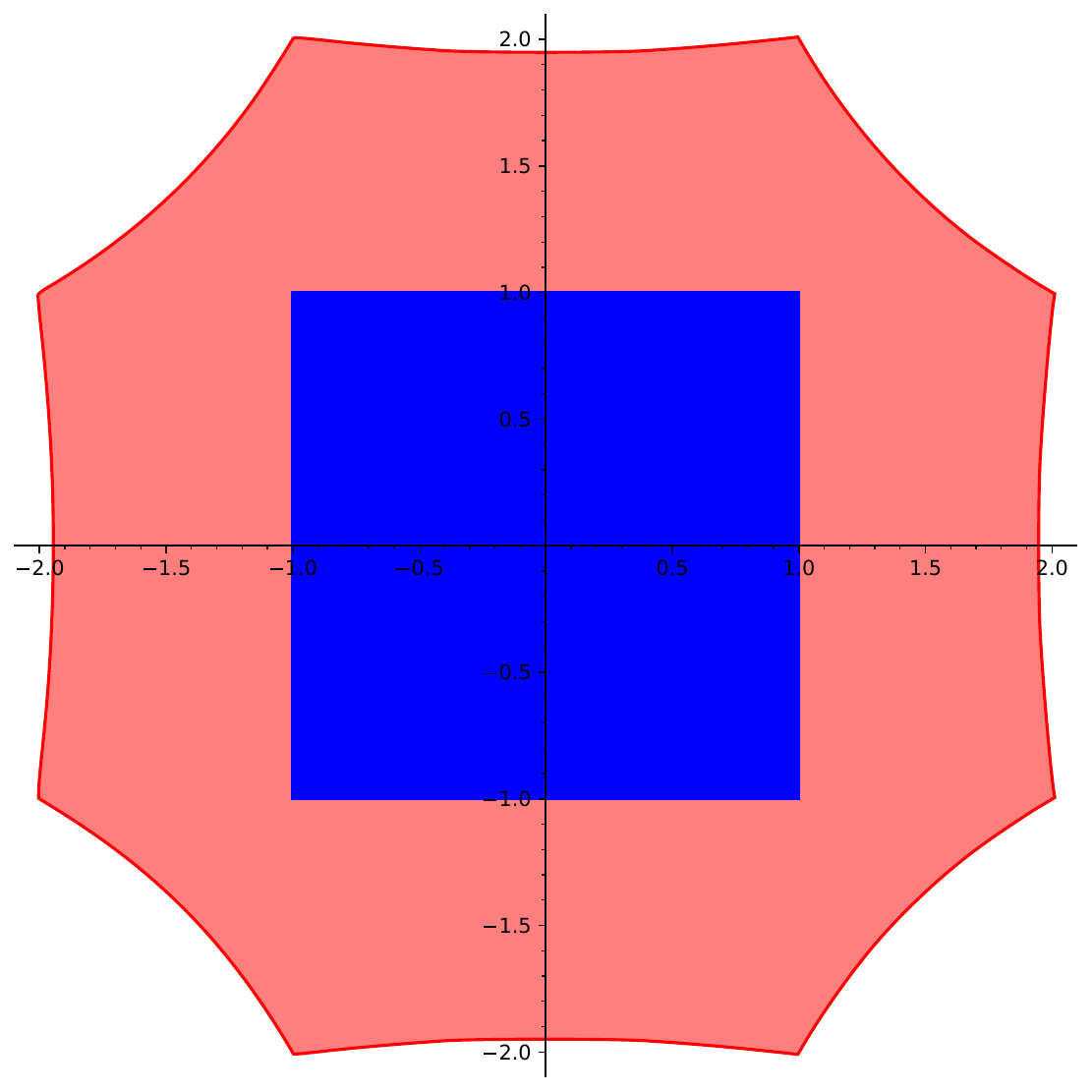}};
        \node at (6,0) {\includegraphics[width=6cm]{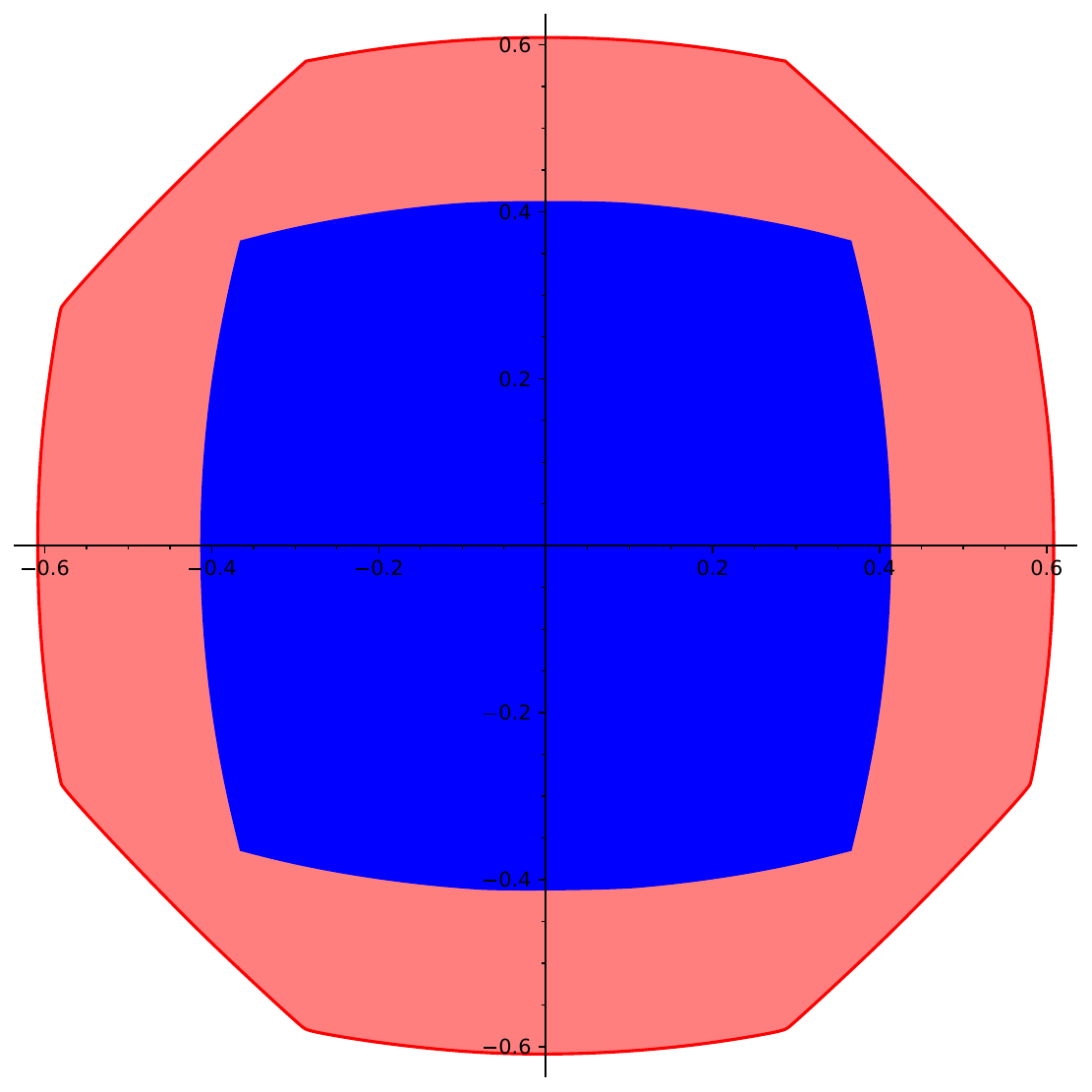}};
    \end{tikzpicture}
\caption{
The spherical illumination body (red and blue part) of a spherical square (blue part) in the gnomonic projection (left) and the stereographic projection (right). The illumination body of the square is not geodesically convex.}
\label{fig:spherical_square}
\end{figure}

We consider only compact geodesically convex subsets, but illumination bodies may also be defined for any spherical domain that is the intersection of closed half-spheres (so even if it may still contains antipodal points in the boundary).

\begin{example}[Illumination body of a spherical wedge]
The intersection of two closed half-spheres $H^+(\bp)$ and $H^+(\bq)$ with centers $\bp$ and $\bq$ in $\SS^n$ is either a $(n-1)$-dimensional great subsphere if $\bp=-\bq$, or a spherical wedge $W(\alpha)$ of width $2\alpha = \pi-d(\bp,\bq)$ for $\alpha\in(0,\frac{\pi}{2}]$ that is symmetric about $\bz = \frac{\bp+\bq}{\|\bp+\bq\|}$. The wedge $W(\alpha)$ is contained in the closed half-sphere about its center $\bz$ and converges to the closed half-sphere $H^+=W(\pi/2)$ for $\alpha\to \pi/2$.
For the spherical volume we find
\begin{equation*}
    \vol_n^s(W(\alpha)) = 2\alpha \kappa_{n-1}.
\end{equation*}

Denote by $\ell$ the $1$-dimensonal great circle spanned by $\bp$ and $\bq$ and denote by $\mathrm{proj}_{\ell}$ the spherical projection to $\ell$, which is uniquely determined for all points $\bz \in \SS^n \setminus (\partial H^+(\bp)\cap\partial H^+(\bq))$.

For points $\bz\in\SS^n$ we distinguish between the following cases:
\begin{enumerate}
 \item[i)] If $\bz\in\SS^n\setminus (H^+(\bp)\cup H^-(\bq))$, then $[\bz, W(\alpha)] = \SS^n$ and therefore $V_{W(\alpha)}^s(\bz) = \omega_n - \vol_n^s(W(\alpha))$.

 \item[ii)] If $\bz\in H^+(\bp)\setminus H^+(\bq)$, then $[\bz, W(\alpha)]$ is again a spherical wedge contained in $H^+(\bp)$ of width $\gamma(\bz)=2\alpha+d(\mathrm{proj}_\ell(\bz),\bq)-\pi/2$.
    Thus
    \[
        V_{W(\alpha)}^s(\bz) = \vol_n^s(W(\gamma(\bz))) - \vol_n(W(\alpha)) = \kappa_{n-1} (d(\mathrm{proj}_\ell(\bz),\bq)-\pi/2)
    \]
 \item[iii)] Analogously, if $\bz\in H^+(\bq)\setminus H^+(\bp)$, then $[\bz,W(\alpha)]$ is a spherical wedge contained in $H^+(\bq)$ of width $2\alpha + d(\mathrm{proj}_\ell(\bz),\bp)$ and $V_{W(\alpha)}^{s}(\bz) = \kappa_{n-1} (d(\mathrm{proj}_\ell(\bz),\bp)$.

 \item[iv)] Finally, for $\bz\in H^+(\bp)\cap H^+(\bq)=W(\alpha)$ we have $[\bz,W(\alpha)]=W(\alpha)$ and therefore $V_{W(\alpha)}^s(\bz)=0$.
\end{enumerate}

Thus, for $\delta\in (0,\delta_\alpha)$ the sublevel sets of $V_{W(\alpha)}^s$ are (possibly degenerated) spherical wedges $W(R(\delta,\alpha))$ with center $\bz$, where
\begin{equation*}
    \delta_\alpha: = \frac{\omega_n}{2} - \vol_n^s(W(\alpha)) = (\pi-2\alpha)\kappa_{n-1},
\end{equation*}
and
\begin{equation*}
    R(\delta,\alpha) = \alpha+\frac{\delta}{\kappa_{n-1}} \in (\alpha,\pi-\alpha).
\end{equation*}
Note that $\pi-\alpha >\pi/2$ and therefore if $R(\delta,\alpha) > \pi/2$, or equivalently if $\delta > \delta_\alpha/2$, then the illumination body $\cI_\delta^s(W(\alpha))$ is not contained in the closed half-sphere $H^+(\bz)$ and therefore not spherical convex, but still star-shaped with respect to $\bz$. For $\delta\to \delta_\alpha^-$ the illumination body $\cI_\delta^s(W(\alpha))$ converges to $H^+(\bp)\cup H^+(\bq)$.
\end{example}

\begin{figure}[t]
\centering
 \begin{tikzpicture}
        \node at (0,0) {\includegraphics[width=6cm]{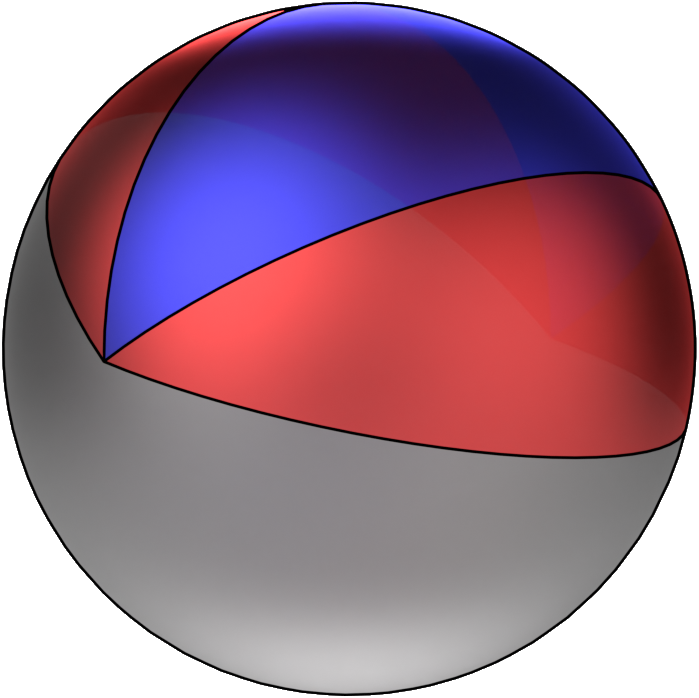}};

    \end{tikzpicture}
\caption{The illumination body (red and blue parts) of a spherical wedge (blue part) on $\SS^2$.}

\label{fig:spherical_wedge}
\end{figure}

\begin{example}[Illumination body of a horoball]
The illumination body of a horoball $B^\infty\subset \HH^n$ is again a horoball (meeting the same ideal point as $B^\infty$) for $\delta \in (0,\delta_1^\infty)$, where
\begin{align*}
    \delta_1^\infty := \vol_n^h([\bu,B^\infty])
    &= \kappa_{n-1} \int_{0}^1 h^{-n} (1-\sqrt{1-h^2})^{n-1} \dd h\\
    &< \kappa_{n-1} \int_{0}^1 h^{n-2}\dd h < +\infty.
\end{align*}
For the hyperbolic plane, i.e.\ $n=2$, we find
\begin{equation*}
    \delta_1^\infty = 2-\frac{\pi}{2} = 0.429203...
\end{equation*}

For the calculation of $\delta_1^\infty$ we may consider the half-space model of $\HH^n$ and consider the horoball that is given by the half-space $B^\infty = \{\bx \in \RR^n : x_n \geq 1\}$ and let $\bu=\bo$ be the origin of $\RR^n$. Then
\begin{align*}
    \vol_n^h([\bo,B^\infty]\setminus B^\infty) &= \int_{[\bo,B^\infty]\setminus B^\infty} \frac{1}{x_n^n} \, \dd \bx \\
    & = \int_0^1 \frac{1}{h^n} \vol_n^e([\bo,B^\infty]\cap \{x_n=h\}) \, \dd h.
\end{align*}
Now the section at height $x_n=h$ of $[\bo,B^\infty]$ is a $(n-1)$-dimensional Euclidean ball of radius $r(h)=1-\sqrt{1-h^2}$ and the expression for $\delta_1^\infty$ follows, see Figure~\ref{fig:1}.
\end{example}

\begin{example}[Illumination body of an ideal triangle]
    An ideal point of the hyperbolic plane $\HH^2$ is a point on the boundary at infinity, that is, a point on the circle $\SS^1\cong\partial \HH^2$. Any three distinct ideal points determine an ideal hyperbolic triangle $T_\infty\subset \HH^2$, which is unbounded and geodesically convex. The interior angles at the ideal vertices of $T_\infty$ are zero and this shows that the hyperbolic area is $\vol_2^h(T_\infty)=\pi$. Note also that all ideal triangles are congruent in $\HH^2$, i.e., up to isometry the ideal triangle is unique. In fact, among all hyperbolic triangles $T$, the ideal triangle $T_\infty$ maximize the hyperbolic area, i.e., $\vol_2^h(T)\leq \vol_2^h(T_\infty)=\pi$.

    For a point $\bz \in \HH^2\setminus T_\infty$ there are two ideal vertices of $T_\infty$ that can be seen from $\bz$, say $\ba,\bb\in\partial \HH^2$. The hyperbolic triangle $T(\bz):=\conv\{\bz,\ba,\bb\}$ has area $\vol_2^h(T(\bz)) = \pi - \vartheta$, where $\vartheta$ is the interior angle at $\bz$ of $T(\bz)$. For a fixed angle $\vartheta$, these points are at a fixed hyperbolic distance $h$ from the side $[\ba,\bb]$ of $T_\infty$, where
    \begin{equation*}
        (\cosh h)(\sin\vartheta/2) = 1.
    \end{equation*}
    Hence these points are on a $\lambda$-geodesic curve with ideal endpoints $\ba,\bb$ where
    \begin{equation*}
        \lambda = \tanh d = \cos \frac{\vartheta}{2} = \sin \frac{\delta}{2},
    \end{equation*}
    for $\delta = V_{T_\infty}^h(\bz) = \vol_2^h(T(\bz))\in (0,\pi)$.

    This yields for $\delta\in (0,\pi)$, that the hyperbolic illumination body $\cI_\delta^h(T_\infty)$ is a $\lambda$-geodesic ideal triangle with the same ideal vertices as $T_\infty$ and where $\lambda=\sin\frac{\delta}{2}\in (0,1)$; also see Figure~\ref{fig:hyperbolic_triangle}.
    For $\delta \to \pi^-$ the area of $T(\bz)$ goes to $\pi$ which forces $\bz$ towards infinity. Therefore $\lim_{\delta\to \pi^-} \cI_\delta^h(T_\infty) = \HH^2$.
\end{example}

\begin{figure}
    \centering
    \begin{tikzpicture}[scale=4]
        \fill[opacity=.2] (0,0) arc (180:90:1) -- (-1,1)  arc(90:0:1) -- cycle;

        \draw[->] (-1.2,0) -- (1.2,0) node[below] {$\be_n^\bot\cong \RR^{n-1}$};
        \draw[->] (0,-0.2) -- (0,1.2) node[left] {$\be_n$};
        \draw[thick] (-1.2,1) -- (1.2,1);
        \draw[thick] (0,0) arc (180:90:1);
        \draw[thick] (0,0) arc (0:90:1);
        \draw[dotted] (1,0) -- (1,1);
        \draw (1,0) -- ({1-sqrt(1-0.7^2)},0.7) node[midway,above] {$1$} -- ({1-sqrt(1-0.7^2)},0) node[midway,right] {$h$};

        \draw (0,0.7) -- ({1-sqrt(1-0.7^2)},0.7) node[midway, above] {$r(h)$};

        \fill (1,1) circle (0.02);
        \fill (-1,1) circle (0.02);
        \fill (0,0) circle (0.02) node[below right] {$\bo$};
        \node[right] at (1.2,1) {$\partial B^\infty$};
        \node at (-0.2,0.9) {$[\bo,B^\infty]\setminus B^\infty$};
    \end{tikzpicture}
    \caption{Sketch for the maximal horoball cap.}
    \label{fig:1}
\end{figure}
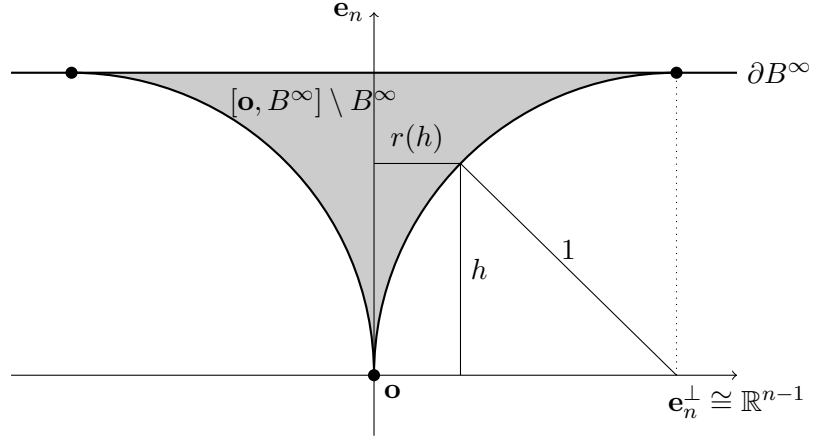

\begin{figure}
    \centering
    \begin{tikzpicture}
        \node at (0,0) {\includegraphics[width=6cm]{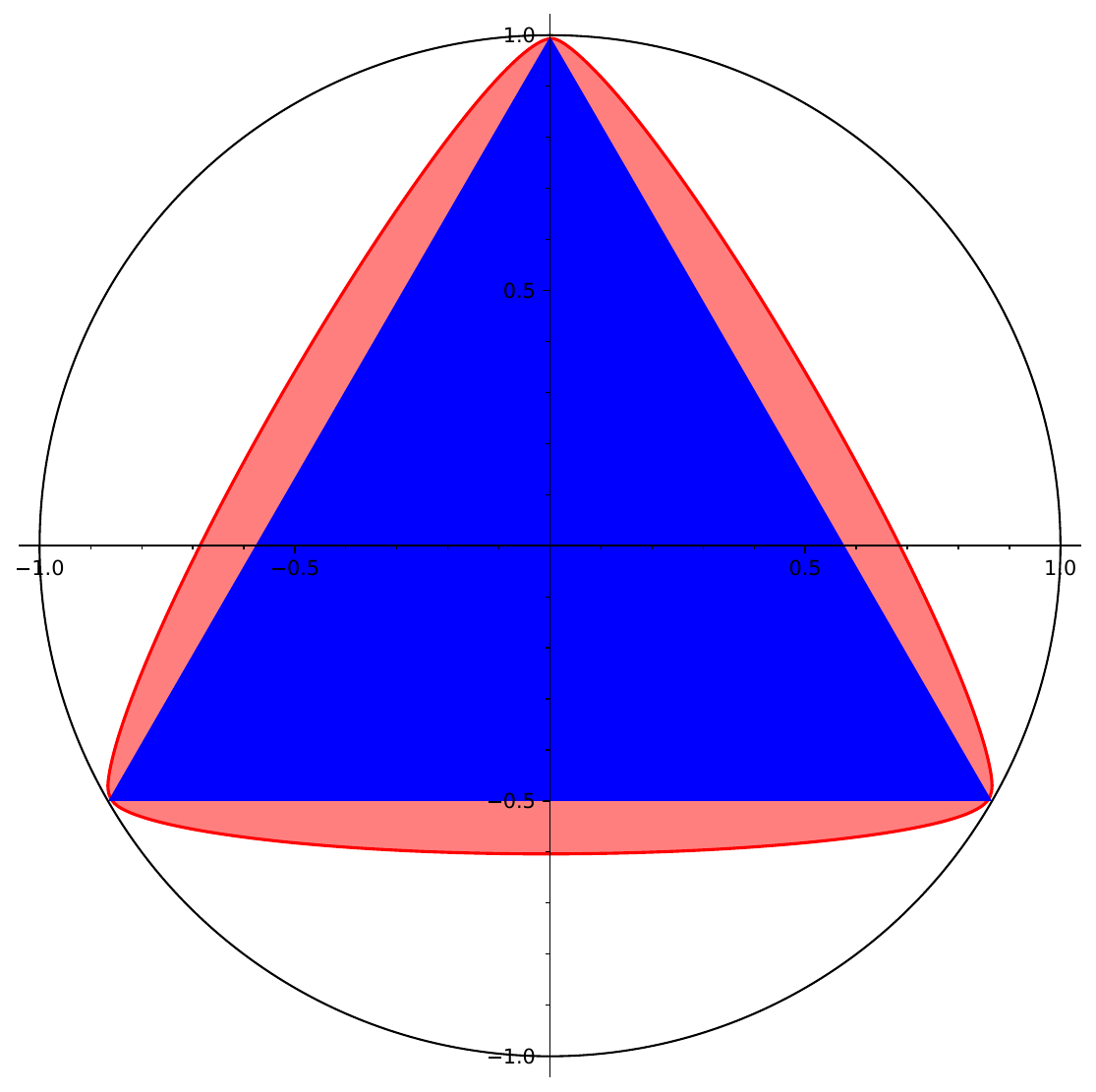}};
        \node at (6,0) {\includegraphics[width=6cm]{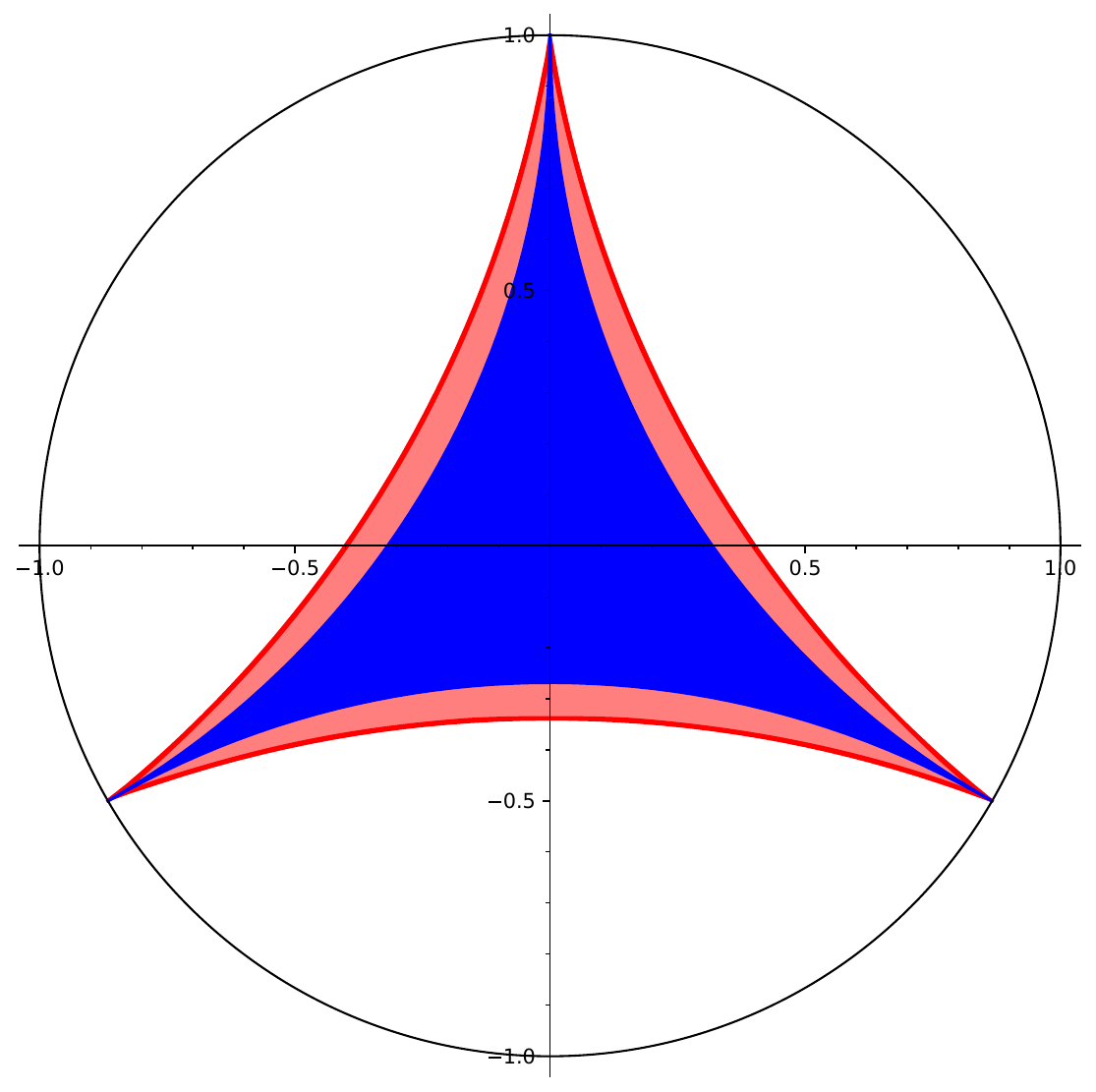}};
    \end{tikzpicture}
    \caption{Hyperbolic illumination body of an ideal triangle in the projective model (left) and the Poincar\'e model (right).
    The hyperbolic illumination body is a $\lambda$-geodesic ideal triangle where $\lambda=\sin \frac{\delta}{2}$ for $\delta\in (0,\pi)$.}
    \label{fig:hyperbolic_triangle}
\end{figure}

\end{document}